\documentclass[3p,twocolumn]{elsarticle}

\newcommand{\showfigs}{1}

\newcommand{\mytikz}[2]
{\ifodd \showfigs
\begin{tikzpicture}[scale=#1] \input{#2} \end{tikzpicture}
\fi
}

\newcommand{\Erdos}{Erd\H{o}s}

\newcommand{\be}{\begin{eqnarray}}
\newcommand{\ee}{\end{eqnarray}}
\newcommand{\ben}{\begin{eqnarray*}}
\newcommand{\een}{\end{eqnarray*}}

\newtheorem{theorem}{Theorem}[section]
\newtheorem{lemma}[theorem]{Lemma}

\newtheorem{corollary}[theorem]{Corollary}

\newenvironment{proof}[1][Proof]{\begin{trivlist}
\item[\hskip \labelsep {\bfseries #1}]}{\end{trivlist}}
\newenvironment{definition}[1][Definition]{\begin{trivlist}
\item[\hskip \labelsep {\bfseries #1}]}{\end{trivlist}}

\renewcommand{\qed}{\nobreak \ifvmode \relax \else
      \ifdim\lastskip<1.5em \hskip-\lastskip
      \hskip1.5em plus0em minus0.5em \fi \nobreak
      \vrule height0.75em width0.5em depth0.25em\fi}

\newenvironment{conjecture}[1][Conjecture]{\begin{trivlist}
\item[\hskip \labelsep {\bfseries #1}]}{\end{trivlist}}

\newcommand{\defin}[1]{\begin{definition} #1 \end{definition}}
\newcommand{\lemm}[1]{\begin{lemma} #1 \end{lemma}}
\newcommand{\theorm}[1]{\begin{theorem} #1 \end{theorem}}

\newcommand{\vertices}[1]{ \foreach \pos/\name in {#1} 
    \node[smallvertex] (\name) at \pos {}; }
\newcommand{\stubs}[1]{ \foreach \pos/\name in {#1} 
    \node[stub] (\name) at \pos {}; }

\newcommand{\edges}[1]{ \foreach \source/ \dest in {#1} \path[edge] (\source) -- (\dest); }
\newcommand{\dotedges}[1]{ \foreach \source/ \dest in {#1} \path[dotedge] (\source) -- (\dest); }

\newcommand{\advgraph}[4]{ \vertices{#1}\stubs{#2} \edges{#3}\dotedges{#4} }

\usepackage{graphicx}
\usepackage{amsmath}
\usepackage{amssymb}

\usepackage{pgfplots}
\usetikzlibrary{pgfplots.groupplots}
\usetikzlibrary{shapes,arrows}

\begin{document}

\tikzstyle{vertex}=[circle, draw, inner sep=0pt, minimum size=5pt]
\tikzstyle{smallvertex}=[circle, draw, inner sep=0pt, minimum size=4pt]
\tikzstyle{blackvertex}=[circle, fill=black, draw, inner sep=0pt, minimum size=4pt]
\tikzstyle{dmvertex}=[diamond, draw, fill, inner sep=0pt, minimum size=4pt]
\tikzstyle{stub}=[circle,  inner sep=0pt, minimum size=1pt]
\tikzstyle{bigvertex}=[circle, draw, inner sep=0pt, minimum size=12pt]
\tikzstyle{edge} = [draw,-]
\tikzstyle{thickedge} = [draw,thick,-]
\tikzstyle{dotedge} = [draw,dotted]

\title{Distance-critical and distance-redundant graphs}

\author{Andrew M. Steane}
\address{
Department of Atomic and Laser Physics, Clarendon Laboratory,\\
Parks Road, Oxford, OX1 3PU, England.
}

\date{\today}

\begin{abstract}
If a vertex in a graph can be deleted without affecting distances
among the other vertices, we shall say it is {\em distance-redundant}. 
Graphs with all, some or no such vertices are discussed. (The latter class
was termed {\em distance-critical} by \Erdos\ and Howorka).
\end{abstract}


\maketitle

\section{Notation and background}

$G(V,E)$ is a simple undirected graph (no self-loops, no repeated edges) 
with vertex-set $V$
and edge-set $E$. 
$u,v,w$ are vertices. $d_u$ is the degree of vertex $u$. $d(u,v)$ is the
distance between $u$ and $v$. $N(u)$ is the open neighbourhood of $u$.
$N[u]$ is the closed neighbourhood of $u$. $G-v$ signifies the graph obtained
by removing vertex $v$ and its associated edges from $G$; (thus 
$G-v$ has vertex set $V\setminus\{v\}$ and edge set $E \setminus \{vw : w \in N(v)\})$).
Such an operation is called either {\em removal} or {\em deletion} of
a vertex (these are synonyms in this paper). The tensor product of
two graphs $A$ and $B$ is denoted $A \otimes B$ and is defined as that
graph whose adjacency matrix is the Kronecker product of those of $A$
and $B$. $K_n$ is the complete graph on $n$ vertices.\cite{Bondy1976,Xu2003}

The background to this work is two-fold. The idea that a vertex may
be superfluous (for some purposes) arises in the study of the Shannon switching 
game on vertices.\cite{05Hayward,06Hayward,Steane2012}
A vertex may not contribute usefully to some task such as forming
a coloured path, for example, because some other nearby vertex does the same and more. 
It occurred to the author that this idea has wider application to the study of distance
between vertices in graphs, and of isometric operations
(i.e. those that leave distances unchanged). The present paper defines a notion of
redundancy and criticality regarding distance in simple graphs, and sets out
a survey of properties related to this. Among other things we will be interested
in those graphs in which it is not possible to remove a vertex without increasing
the distance between some other pair of vertices. These graphs were termed
{\em distance-critical} by \Erdos\ and Howorka, who gave limits on an upper bound on 
the number of edges.\cite{Erdos1980} This term is not widely used, however, which suggests
it has not been widely studied. The diameter-critical graphs, by contrast, have
received much attention.\cite{Loh2016,Fan1986,Erdos1962}


The present discussion presents
many small results as opposed to a few large ones, but some that are easy to prove
were not easy to think of, and some are surprising.
Section \ref{s.surround} presents a relation leading to a partial ordering of vertices
which is invoked in some of
the later arguments. Section \ref{s.dr} defines {\em distance redundancy} and
obtains some basic facts. Section \ref{s.ws} gives an introduction to graphs which
are critical with respect to this property. The later sections comment on
random graphs, generation, tensor products, enumeration and on the idea of
recursively pruning the redundant vertices so as to obtain an isometric smaller
graph. I don't know if any of this will be of practical importance, but the basic idea
is quite simple and natural, which suggests it may contribute to other studies
in graph theory.

\section{On the `surrounds' relation} \label{s.surround}

\defin{
Vertex $v$ {\em surrounds} vertex $u$ iff $N(u) \subset N[v]$.
}

\defin{
Vertices $u$ and $v$ are {\em twins} if
either $N(u) = N(v)$ (called {\em weak twins}) or
$N[u] = N[v]$ (called {\em strong twins}).
}

In the literature there is also the terminology `weak and strong siblings' and
the terminology `false and true twins' referring to this same concept. 
When we assert that a vertex is a twin we allow that there may be more than one other vertex having the same neighbourhood (thus a set of twins can be of any
size greater than 1). According to our terminology, a pair (or more) of 
isolated vertices are twins.\footnote{In the literature, a graph without weak twins is called a {\em mating-type graph} or a 
{\em point determining} graph.\cite{Gessel2011}} 
The term `surrounds', by contrast, is not currently established; it is
introduced here and in \cite{Steane2012}.

\begin{lemma} \label{surroundsimple}
(Simple observations):
\begin{enumerate}
\item A simplicial vertex is surrounded by all its neighbours.

\item If $u$ and $v$ surround each other then they are twins. 

\item A surrounded vertex cannot be a cut-vertex (also called articulation
vertex). 

\item If $u$ is surrounded by a vertex of the same degree then they are twins. 

\item If $v$ surrounds $u$ in $G$ then $u$ surrounds $v$ in 
the complement graph $\bar{G}$.
\end{enumerate}
\end{lemma}

\begin{proof}
The first two are easy to prove. For the third, observe that if
$u$ is surrounded by $w$ then for every path $s$-$u$-$v$ (for vertices $s$,$u$,$v$)
there is a path 
either $s$-$v$ or $s$-$w$-$v$ so $s$ and $v$ are still in the same connected
component after deletion of $u$. \qed

For item (iv),
let $v$ be the surrounding vertex, and let $d_u$, $d_v$ be the degrees of $u$ and $v$
respectively. If $v$ is not adjacent to $u$ then $|N(u)| = d_u$ of its edges go to
neighbours of $u$. If $d_v = d_u$ then $v$ has no other edges, therefore $u$ is
adjacent to all of $N(v)$, therefore $u$ surrounds $v$ so they are twins. 
If $v$ is adjacent to $u$ then $|N(u)|-1 = d_u-1$ of its edges go to $N(u)$
and one goes to $u$. If $d_v = d_u$ then it has no other edges, therefore
$u$ is adjacent to all of $N[v]$, therefore $u$ surrounds $v$ so they are twins.
\qed

For item (v), observe that since there is an edge between $v$
and every neighbour of $u$ in $G$, there must be a non-edge between $v$ and
every non-neighbour of $u$ in $\bar{G}$. Hence in $\bar{G}$ all edges from
$v$ go to neighbours of $u$, hence $u$ surrounds $v$. \qed
\end{proof}

It follows that $G$ has no surrounded vertices iff $\bar{G}$ has
no surrounded vertices. Hence graphs with no surrounded vertices come in
pairs whenever $G \ne \bar{G}$, and so do graphs with all surrounded vertices.

{\bf Notation}:
The notation $v \succ u$ shall signify that $v$ surrounds $u$ and is
not surrounded by $u$. The notation $v \succeq u$ shall signify that 
$v$ surrounds $u$ and may or may not be surrounded by $u$. The notation
$v \cong u$ shall signify that $u$ and $v$ are twins. The notation
$v \prec u$ is equivalent to $u \succ v$. 

\lemm{The surrounding relation is a partial ordering.} 

\begin{proof} We show the three required properties.
(1) $u \succeq u$ by the definition of `surrounds'. (2) If $u \succeq v$ and
$v \succeq u$ then $u \cong v$; this was already mentioned 
in item (ii) of lemma \ref{surroundsimple}. (3)
If $w \succeq v$ and $v \succeq u$ then $w \succeq u$; this is obvious
from the definitions. \qed
\end{proof}

Some simple consequences.

If $u \succeq v$ and $v \succeq w$ and $w \succeq u$ then $u \cong v \cong w$.

It follows that in the directed graph of surrounding relations (that is,
the directed graph in which there is an arc from $u$ to $v$ iff $u \succeq v$
in $G$) all cycles can be traversed in either direction (in other words,
there are no cycles without this property). Also, if all vertices in a graph are surrounded then there must be at least two twins.

\lemm{  \label{cutsurrounded}
(i)
If $w \succeq v \succeq s$ in $G$, for distinct vertices $s,v,w$,
then $w \succeq v$ in $G-s$.\\
(ii)
If $w \succeq v$ and $u \succeq s$ in $G$,
for distinct vertices $u,v,w,s$, then $w \succeq v$ in $G-s$.
}
\begin{proof}
The only edges affected by the deletion of $s$ are those incident
on $s$. It follows that if $w$ is not a neighbour of $s$ then deletion
of $s$ does not change $N(w)$. Meanwhile the deletion of $s$ can only
remove, not add, members to $N(v)$, so $v$ will still be surrounded by
$w$. In the case where $w$ is a neighbour of $s$ then $N(w)$ will lose
one member, but this same member (namely $s$) either was not present in
$N(v)$ or is also removed from $N(v)$, so $w$ will still surround $v$.
\qed
\end{proof}

\defin{
A graph operation is termed {\em isometric} if it preserves the distances among
vertices present before and after the operation.}

(According to the terminology adopted here, the deletion of a cut-vertex
is not an isometric operation.)

\theorm{ \label{th.isosurround}
(i) Removing all but one of each set of twins is isometric.\\
(ii) Removing all non-twin surrounded vertices is isometric.
}
 
\begin{proof}
(i) is obvious but if the reader desires further demonstration,
see the observations on distance-redundant vertices made in the rest of the paper. To prove (ii) we first consider the removal of a single vertex $s$.
If $v \succ s$ then for every path of least distance between vertices
$u \ne s$ and $w \ne s$ involving the walk $nsm$ (where $\{n,m\} \in N(s)$)
there is another path of the same length where this section is replaced by
$nvm$. It follows that the removal of $s$ is isometric. 
The case of multiple surrounded vertices is treated by induction
employing lemma \ref{cutsurrounded}, as follows.
Let $S$ be the set of non-twin surrounded vertices in $G$.
Pick any two members ($s,v$) of $S$. Since they are not twins
the only possibilities are that one surrounds the other or that
neither surrounds the other. These possibilities are the ones considered
in the two parts of lemma \ref{cutsurrounded}. It follows that we can
always remove one of them from $G$ (which we already proved to be
isometric) while leaving the other still surrounded. Continuing thus
we can eventually remove all of $S$ by isometric operations.
\qed
\end{proof}

In such a sequence of deletions, one or more vertices in
$S$ that are not twins in $G$ may become twinned with other vertices
in the subgraph obtained by deletion. 
However the vertex/vertices with which they are twinned
will not themselves be in $S$ so it never happens that all members of
a set of twins are deleted.

\begin{corollary}
Every cut-set contains at least one vertex which is either
a twin or not surrounded.
\end{corollary}


\begin{theorem} \label{th.stensor}
	If two graphs each have no surrounded vertex then their tensor
	product has no surrounded vertex.
\end{theorem}
\begin{proof}
	Let $A$, $B$ be the two graphs and $T = A \otimes B$. 
	Consider the vertex $(u,v)$ in $T$, where $u \in A$, $v \in B$.
	Some other vertex $(s, t)$ surrounds $(u,v)$ iff it
	is adjacent to all the neighbours of $(u,v)$ that are not itself.
	Those neighbours
	are the vertices $(u_i, v_j)$ where $u_i \in N(u)$ and
	$v_j \in N(v)$. 
	In order that $(s,t)$ shall surround $(u,v)$, 
	therefore, we require that $s$ is adjacent to all of $N(u) \setminus \{s\}$ 
	and $t$ is adjacent to all of $N(v) \setminus \{t\}$. This can happen only
	in the following situations:
	($s=u$ and $t \succeq v$) or ($t=v$ and $s \succeq u$)
	or ($s \succeq u$ and $t \succeq v$).
	However if neither graph has any surrounded vertex
	then none of these cases can arise for any $(u,v)$, hence
	no vertex in $T$ is surrounded.
	\qed
\end{proof}

It is possible for a tensor product of graphs with all
surrounded vertices to have no surrounded vertices; an example is
$K_3 \otimes K_3$. 

\ifodd 0
\begin{conjecture}
[Tanya Khovanova]
The number of unlabelled point-determining graphs of given order is equal to the number
of unlabelled endpoint-free graphs of that order.
\end{conjecture}

This conjecture is included merely for note; it is not otherwise connected
to the discussion in this paper. Its plausibility is owing to the fact that it
is true for orders up to 10. The sequence of values is OEIS A004110.
[I NOW THINK THIS IS ALREADY KNOWN AND PROVED]
\fi

{\bf Algorithms}.
{\em A simple matrix method to find surrounding--surrounded vertex
pairs.} For a matrix $M$ let $[M]_{j,k}$ denote the element at
row $j$, column $k$. 
Let $a$ be the adjacency matrix of a graph $G$.
Then $v \succeq u$ in $G$ iff $[ a^2 + a ]_{u,v} = d_u$.
(The matrix $a^2$ can be obtained from $a$ without a matrix multiplication
by regarding each row of $a$ as a single binary number; one
combines these binary numbers in pairs by bitwise {\sc and} operations,
recording the Hamming weight (sum of binary digits) of the results.)

To find twins one may use lemma \ref{surroundsimple}(iv). Another
simple way is to regard each row of the
adjacency matrix as a single binary number. One sorts these
binary numbers (keeping track of which is which) and then
looks for equal adjacent entries in the sorted list. This reveals
the weak twins. To find the strong twins, perform the same operation
but starting from the matrix $a + I$ where $I$ is the identity matrix.

\section{Distance-redundancy} \label{s.dr}

\defin{
A vertex is {\em first-order distance redundant} iff its removal does not change any distance among the other vertices.}

\defin{
A vertex $u$ is {\em $k$'th-order distance redundant} iff its removal does not change the distance between any
pair of vertices $\{ v, w \}$ satisfying $d(u,v) \ge k$, $d(u,w) \ge k$, $d(v,w) \ge 2k$.}

(The first-order definition here is subsumed by the general definition;
it has been included for clarity.)

The idea behind these definitions is
that deletion of a vertex may affect distances among nearby vertices without affecting distances among
those further away. It follows, for example, that if we wish to learn the distance between two vertices
$\{ t_1, t_2 \}$ and we have found a first-order distance redundant vertex $u \notin \{t_1,t_2\}$
then we could first delete $u$ and then find $d(t_1, t_2)$ in the resulting smaller graph. Equally,
if we have established that $d(t_1, t_2) \ge 4$ then we could remove a vertex $u$ which is
2nd-order distance redundant and not in $N[t_1] \cup N[t_2]$
and then find
$d(t_1,t_2)$ in the resulting smaller graph.

In the following we use the unadorned term `redundant' as a shorthand for 
first-order distance redundant.

\subsection{Properties associated with redundant vertices}

A redundant vertex has near to it some set of vertices
providing 2-walks between its neighbours;
in the case of a surrounded vertex there is such a set with just one
member.

\lemm{
(Simple observations):
\begin{enumerate}
\item A surrounded vertex is redundant.

\item A cut-vertex is not redundant. 

\item If a vertex is $k$'th-order redundant then it is also
$(k+1)$'th-order redundant.

\item A vertex of degree $2$ is redundant iff it is surrounded.

\item Every vertex in a graph of diameter $D$ is $\lceil (D+1)/2 \rceil$-order
redundant.
\end{enumerate}
}

(i)--(iv) are straightforward. Item (v) follows from the observation
that there are no vertices at distances $\ge 2k$ when $2k$ exceeds the diameter.

\begin{figure}
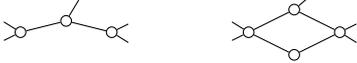

	\centering{
\tikz[scale=0.3]{
\advgraph    
{ {(-2,0)/a}, {(0,0.5)/b}, {(2,0)/c}}
{ {(-2.8,0.5)/as}, {(-2.8,-0.4)/bs}, {(2.8,0.4)/cs}, {(2.8,-0.5)/ds},
{(0.6,1.5)/es}}    
{a/b, b/c, a/as, a/bs, c/cs, c/ds, b/es}
{}
\advgraph    
{ {(8,0)/a}, {(10,1)/b}, {(10,-1)/c}, {(12,0)/d}}
{ {(7.2,0.5)/as}, {(7.2,-0.4)/bs}, {(12.8,0.4)/cs}, {(12.8,-0.5)/ds},
{(10.7,1.6)/es}}    
{a/b, a/c, b/d, c/d, a/as, a/bs, d/cs, d/ds, b/es}
{}
}
}
\caption{Illustration of weak link (left) and strong link (right).}
\label{f.wkstrong}
\end{figure}

\defin{A pair of vertices is {\em weakly linked} iff they
are not adjacent and there is one and only one 2-walk between them.
(c.f. Fig. \ref{f.wkstrong}.)
A pair of vertices is {\em strongly linked} 
iff they are either adjacent or have more than one 2-walk between them.
Thus $u$ and $v$ are strongly linked iff either $v \in N(u)$ or
$|N(u) \cap N(v)| \ge 2$.
}

A useful way to think about first-order redundant vertices
is the following. Consider a vertex $s$. 
We take the subgraph
induced by $N(s)$, and add to it edges between any vertices which
are strongly linked in the original graph. Iff the result 
is a complete graph, $s$ is redundant (see Fig. \ref{f.clique}).

Alternatively, one may take a given graph $G$ and then construct alongside
it a graph $G'$ whose vertices correspond to those of $G$ and which has no edges
in the first instance. Then add to $G'$ an edge between any pair of vertices
for which the corresponding pair in $G$ is weakly linked. A vertex $v$
in $G$ is redundant iff there are no edges in $G'$ among the vertices
corresponding to $N(v)$.

\begin{figure}
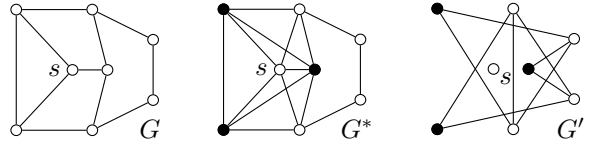

 \tikz[scale=0.2]{
\vertices
{ {(2,-2)/a}, {(2,2)/b}, {(-2,4)/c}, {(-7,4)/d}, {(-7,-4)/e}, {(-2,-4)/f}, {(-1,0)/g}, {(-3.3,0)/h}}
\edges
{a/b, a/f, b/c, c/d, c/g, d/e, d/h, e/f, e/h, f/g, g/h}
\path (-4.5,0) node[] {$s$};
\path (1.8,-3.9) node[] {$G$};
}
\hspace{4mm}
 \tikz[scale=0.2]{
\vertices
{ {(2,-2)/a}, {(2,2)/b}, {(-2,4)/c}, {(-7,4)/d}, {(-7,-4)/e}, {(-2,-4)/f}, {(-1,0)/g}, {(-3.3,0)/h}}
\edges
{a/b, a/f, b/c, c/d, c/g, d/e, d/h, e/f, e/h, f/g, g/h, f/h, e/g, c/h, d/g}
\path (-4.5,0) node[] {$s$};
\path (-7,4) node[blackvertex] {};
\path (-7,-4) node[blackvertex] {};
\path (-1,0) node[blackvertex] {};
\path (1.8,-3.9) node[] {$G^*$};
}
\hspace{4mm}
 \tikz[scale=0.2]{
\vertices
{ {(2,-2)/a}, {(2,2)/b}, {(-2,4)/c}, {(-7,4)/d}, {(-7,-4)/e}, {(-2,-4)/f}, {(-1,0)/g}, {(-3.3,0)/h}}
\edges
{a/c, a/e, a/g, b/g, b/f, b/d, c/f, c/e, f/d}
\path (-2.5,-0.6) node[] {$s$};
\path (-7,4) node[blackvertex] {};
\path (-7,-4) node[blackvertex] {};
\path (-1,0) node[blackvertex] {};
\path (1.8,-3.9) node[] {$G'$};
}
\caption{Illustration of redundancy. The left diagram shows a graph $G$.
The middle diagram shows $G^*$ which is given by $G$ with edges added between strongly linked
vertices. The neighbours of vertex $s$ in $G$ are indicated by filled symbols in $G^*$; the subgraph of $G^*$ induced by these vertices is a complete graph.
The right diagram shows $G'$, a graph in which each edge indicates a weak link in $G$. Here the filled vertices induce an empty graph.
}
\label{f.clique}
\end{figure}

\lemm{ \label{l.isored}
If $u$ and $v$ are redundant in $G$ and $d(u,v) > 2$, then
$v$ is redundant in $G-u$.}
\begin{corollary}
If $S$ is a set of redundant
vertices whose members are all at distance $>2$ from each other,
then deletion of all the vertices in $S$ is isometric.
\end{corollary}


\begin{proof}
The deletion of $u$ only removes
from $G$ edges incident on $u$. None of these edges is incident on a
neighbour of $v$ because $d(u,v) > 2$. Therefore no neighbour of $v$
loses an edge when $u$ is deleted. 
Therefore any pair of those neighbours which were strongly linked 
in $G$ remain strongly linked in $G-u$. Hence they are all strongly
linked, hence $v$ is redundant in $G-u$. The corollary then
follows. \qed
\end{proof}

Let $N_k(v)$ be the set of vertices at distance $k$ from $v$. Let
$w_k(u,v)$ be the number of $k$-walks between vertices $u$ and $v$. 

\lemm{
Vertex $s$ is $k$'th-order redundant iff $w_{2k}(u,v) \ne w_k(u,s) w_k(s,v)$
$\forall \{ u,v : u \in N_k(s), v \in N_k(s), d(u,v) = 2k$\}.
{\em 
(For example, $s$ is 1st-order redundant iff the number of 2-walks
between each non-adjacent pair of its neighbours is not equal to~1.)}
}


\begin{proof}
Vertex $s$ is redundant iff its deletion will increase the distance
between two members of $N_k(s)$ which are separated by a distance $2k$
(the distance between them cannot be greater than $2k$ or
they would not both be in $N_k(s)$). This increase will happen iff
all walks of length $2k$ between two such vertices pass through $s$.
For any pair of vertices $u,v$ the number of walks of length $2k$ between 
$u$ and $v$ via $s$ is equal to $\sum_{t=1}^{2k-1} w_t(u,s) w_{2k-t}(s,v)$
(since one must first walk from $u$ to $s$ and then from $s$ to $v$).
The conditions of the lemma stipulate that $w_t(u,s) = 0$ for
$t < k$ and $w_{2k-t}(s,v) = 0$ for $2k-t < k$. Therefore the only non-zero
term in the sum is $w_k(u,s) w_{k}(s,v)$.
Now $w_{2k}(u,v)$ is by definition the total number of walks of length $2k$
between $u$ and $v$. Therefore all walks of length $2k$ from $u$ to $v$ pass via
$s$ iff $w_{2k}(u,v) = w_k(u,s) w_{k}(s,v)$. Hence if this equality does
not hold then either there are fewer or more walks. But the
conditions stipulate that there are at least this number, because $u$ and
$v$ are at distance $k$ from $s$. Hence there is one or more walk (between
$u$ and $v$) of length $2k$ not involving $s$ iff the inequality stipulated 
in the lemma holds. \qed
 \end{proof}

\section{Weak and strong graphs}  \label{s.ws}

We now discuss those graphs in which either all vertices are redundant
or none are. We shall use the terminology `strongly linked graph' for a graph
in which all vertices are redundant (reflecting the role of strong links
in redundancy) and `weakly linked graph' for a graph in which no
vertices are redundant (reflecting the role of weak links in such graphs).
We will normally shorten this terminology to `strong graph' and `weak graph'.
Neither a strong graph nor a weak graph need necessarily be connected 
(that is, have just one connected component) but weak graphs have
no isolated vertices.

Most graphs are neither weak nor strong.
Trees with more than two vertices are neither weak nor strong.
A vertex transitive graph must be either weak or strong.

\subsection{Strong graphs}

\begin{lemma} \label{l.strong}
A graph has all vertices redundant iff it has no weakly linked
pairs of vertices.
\end{lemma}

\begin{corollary}
A graph has all vertices redundant iff 
no element of $(a^2 + 2(a+I))$ is equal to 1, where $a$ is the adjacency matrix
and $I$ the identity matrix. 
\end{corollary}

\begin{proof}
If there were a weakly linked pair then the vertex adjacent to 
that pair would not be redundant; if there is no weakly linked pair
then every vertex neighbourhood is strongly linked, therefore every
vertex is redundant. The corollary is straightforward. \qed
\end{proof}

\begin{lemma} \label{str_simple}
In a graph of order $n$ 
with all vertices redundant, (a ``strong graph"):
\begin{enumerate}
\item There are no pendants (vertices of degree 1) except in
components equal to P2.
\item For all $u$, any vertex $v \notin N(u)$ which is
adjacent to a member of $N(u)$ is also adjacent to another member
of $N(u)$. 
\item Neighbours of a degree-2 vertex are twins.
\item If the graph is connected then the diameter is at most $\lfloor n/2 \rfloor$.
\end{enumerate}
\end{lemma}


\begin{proof}
(i) If there were a pendant its neighbour would not be redundant (except
in the graph P2). 
(ii) Is so because otherwise there would be a weak link, such that the
member of $N(u)$ adjacent to $v$ would not be redundant (c.f. also
lemma \ref{l.strong}).
(iii)
Let $v,w$ be the neighbours of a degree-2 vertex $u$ and suppose $v$
has no twin and degree $d_v > 1$.
In this case one of $v$ and $w$ is not redundant because
either there will be a weak
link between $u$ and a neighbour of $v$ not in $N[w]$, or $v$ itself
has degree 2 with its other edge going to $w$. But in the latter case
$w$ is a cut-vertex if it has degree above 2, which cannot happen in a strong graph. The alternative is that $\{u,v,w\}$ forms a component $K_3$ but in that case
$v$ and $w$ are twins. 
(iv) Consider a pair of vertices at the maximal distance and let $p_1$ be
the set of vertices on a shortest path between them, excluding the vertices
themselves (thus for diameter $D$, $|p_1| = D-1$).  
To ensure that vertices on this path are
not redundant, we require that for each $v \in p_1$ there is another
shortest path which does not include $v$. This requires that the set of
shortest paths (between the given vertices) satisfies $p_1 \cap p_2 \cap
\ldots p_k = \emptyset$. But this implies $|p_1 \cup p_2 \cup \ldots p_k|
\ge 2 |p_1|$. It follows that $n \ge 2 |p_1| + 2 = 2D$.
\qed
\end{proof}



\begin{figure}
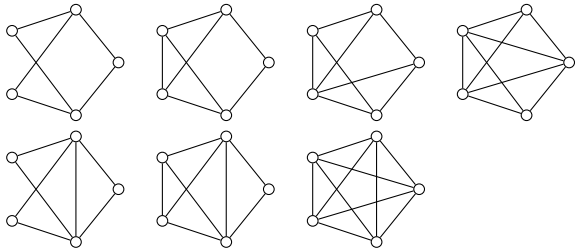

\centering{
\begin{tabular}{cccc}
\tikz[scale=0.2]{\vertices{{( 7.7, 4.2)/a},{( 0.7, 2.1)/b},{( 0.7, 6.3)/c},
{( 4.9, 7.7)/d},{( 4.9, 0.7)/e}}                                           
\edges{a/d,a/e,b/d,b/e,c/d,c/e}                                            
}                                                                          
 & 
\tikz[scale=0.2]{\vertices{{( 4.9, 7.7)/a},{( 4.9, 0.7)/b},{( 7.7, 4.2)/c},
{( 0.7, 6.3)/d},{( 0.7, 2.1)/e}}                                           
\edges{a/c,a/d,a/e,b/c,b/d,b/e,d/e}                                        
}                                                                          
 & 
\tikz[scale=0.2]{\vertices{{( 4.9, 7.7)/a},{( 7.7, 4.2)/b},{( 0.7, 6.3)/c},
{( 0.7, 2.1)/d},{( 4.9, 0.7)/e}}                                           
\edges{a/b,a/c,a/d,b/d,b/e,c/d,c/e,d/e}                                    
}                                                                          
 & 
\tikz[scale=0.2]{\vertices{{( 7.7, 4.2)/a},{( 4.9, 7.7)/b},{( 0.7, 6.3)/c},
{( 4.9, 0.7)/d},{( 0.7, 2.1)/e}}                                           
\edges{a/b,a/c,a/d,a/e,b/c,b/e,c/d,c/e,d/e}                                
}                                                                          
\\ 
\tikz[scale=0.2]{\vertices{{( 4.9, 7.7)/a},{( 7.7, 4.2)/b},{( 0.7, 6.3)/c},
{( 0.7, 2.1)/d},{( 4.9, 0.7)/e}}                                           
\edges{a/b,a/c,a/d,a/e,b/e,c/e,d/e}                                        
}                                                                          
 & 
\tikz[scale=0.2]{\vertices{{( 0.7, 2.1)/a},{( 4.9, 7.7)/b},{( 0.7, 6.3)/c},
{( 4.9, 0.7)/d},{( 7.7, 4.2)/e}}                                           
\edges{a/b,a/c,a/d,b/c,b/d,b/e,c/d,d/e}                                    
}                                                                          
 & 
\tikz[scale=0.2]{\vertices{{( 7.7, 4.2)/a},{( 4.9, 7.7)/b},{( 0.7, 6.3)/c},
{( 0.7, 2.1)/d},{( 4.9, 0.7)/e}}                                           
\edges{a/b,a/c,a/d,a/e,b/c,b/d,b/e,c/d,c/e,d/e}                            
}                                                                          
 & 
\end{tabular}

}
\caption{Connected strong graphs of order 5. All can be constructed
as described in the proof of lemma \ref{l.strongsub}; the first four
using weak twins, the last three using strong twins.}
\label{f.strong5}
\end{figure}

\begin{figure}
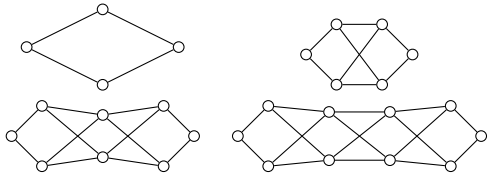

\centering{
\begin{tabular}{cc}
\tikz[scale=0.2]{\vertices{{( 11, 3)/a},{( 6, 5.5)/b},{( 1, 3)/c},
{( 6, 0.5)/d}}                                                    
\edges{a/b,a/d,b/c,c/d}                                           
}                                                                 
 & 
\tikz[scale=0.2]{\vertices{{( 7, 8)/a},{( 7, 4)/b},{( 4, 4)/c},
{( 2, 6)/d},{( 4, 8)/e},{( 9, 6)/f}}                           
\edges{a/c,a/e,a/f,b/c,b/e,b/f,c/d,d/e}                        
}                                                              
 \\ 
\tikz[scale=0.2]{\vertices{{( 4, 10)/a},{( 4, 6)/b},{( 8, 9.4)/c},
{( 12, 10)/d},{( 8, 6.6)/e},{( 2, 8)/f},{( 12, 6)/g},{( 14, 8)/h}}
                                                                  
\edges{a/c,a/e,a/f,b/c,b/e,b/f,c/d,c/g,d/e,d/h,e/g,g/h}           
}                                                                 
 & 
\tikz[scale=0.2]{\vertices{{( 4, 10)/a},{( 4, 6)/b},{( 8, 9.6)/c},    
{( 12, 9.6)/d},{( 8, 6.4)/e},{( 2, 8)/f},{( 12, 6.4)/g},{( 16, 10)/h},
{( 16, 6)/i},{( 18, 8)/j}}                                            
\edges{a/c,a/e,a/f,b/c,b/e,b/f,c/d,c/g,d/e,d/h,d/i,e/g,g/h,g/i,       
h/j,i/j}                                                              
}                                                                     
\end{tabular}

}
\caption{Strong graphs with diameter $n/2$ and no strong twins.
There is only one such graph at each even $n$. The first four of an 
infinite sequence are shown.}
\label{f.strlink}
\end{figure}

The connected strong graphs on 5 vertices are shown in figure \ref{f.strong5}.
When considering these and larger graphs, one wishes to know
what generic properties they have (beside the defining property). 
For example, are there any excluded subgraphs? The answer to this
is no, as the following construction shows:

\begin{lemma} \label{l.strongsub}
\begin{enumerate} 
\item Every graph of order $n$
is an induced subgraph of a strong graph of order $n+2$.
\item Every graph of order $n$ with $k$ non-redundant vertices
is an isometric induced subgraph of some strong graph of order $\le n+k$.
\end{enumerate}
\end{lemma}

\begin{proof}
(i)
To make a strong graph starting from any graph $G$,
add two further vertices, making each of them adjacent to all the vertices in
$G$. The resulting graph is strong since the two added vertices
surround all the other vertices and each other. 
(ii) To make a strong graph of which $G$ is an isometric induced subgraph,
split each non-redundant vertex into a pair of twins. For each vertex
$v$ involved in this operation, the resulting
twins are clearly redundant, and any $w \in N(v)$ which was
redundant remains so, because for each 2-walk from 
a neighbour of $w$ to $v$ there will now also be a 2-walk to the
newly introduced twin of $v$. But since $w$ was redundant those
2-walks come in pairs (making strong links) where necessary, 
and therefore so will the new 2-walks. 
This suffices to guarantee that
the new weak links come in pairs, making strong links, such that
$w$ remains redundant.
\qed
\end{proof}

All the graphs shown in figure \ref{f.strong5} can be generated
from the four graphs of order 3 by using part (i) of this construction. 
This method produces graphs of low diameter---at most 2. The method of
part (ii) requires further vertices but preserves the diameter.

Let us now focus attention on graphs of diameter greater than 2.
These first appear at order 6. Figure
\ref{f.strlink} shows the first four of an infinite sequence of graphs of even
order having diameter $n/2$---the maximum possible (lemma \ref{str_simple}iv). This structure is well known in the study
of the Shannon game where it is sometimes called a {\em chain}.\cite{06Hayward} 
If the end vertices are coloured and then two players take
successive turns, one colouring and the other deleting one
uncoloured vertex per turn, then for graphs of this kind the colourer
can guarantee to colour a complete path between the ends 
whether they play first or second.

Observe next that the addition or removal of an edge between
twins does not change any distance in the
graph except that between those vertices (changing it from 
2 to 1 or 1 to 2).
Therefore for every strong graph containing weak twins there is
a further strong graph containing strong twins with only moderately
different properties. It follows that to understand the structure of
strong graphs in general it is sufficient to understand the structure
of strong graphs having no strong twins. 

A numerical survey revealed that 
the fraction of all connected graphs of given order which are
strong falls approximately exponentially with $n$, for $n < 12$,
falling below
$10\%$ at $n=9$. The fraction of connected graphs with diameter above 2
which are strong falls below $1\%$ at $n=10$. In this sense strong
graphs are rare (but not nearly as rare as weak graphs). For
further information on their abundance see section \ref{s.gnp}.

\subsubsection{Redundant and not surrounded}

\begin{figure}
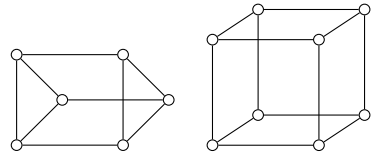

\centering{
\begin{tabular}{cc}
\tikz[scale=0.2]{\vertices{{( 9, 2)/a},{( 2, 2)/b},{( 12, 5)/c},
{( 2, 8)/d},{( 9, 8)/e},{( 5, 5)/f}}                            
\edges{a/b,a/c,a/e,b/d,b/f,c/e,c/f,d/e,d/f}                     
}                                                               
 & 
\tikz[scale=0.2]{\vertices{{( 0, 1)/a},{( 0, 8)/b},{( 7, 1)/c}, 
{( 7, 8)/d},{( 3, 3)/e},{( 3, 10)/f},{( 10, 3)/g},{( 10, 10)/h}}
                                                                
\edges{a/b,a/c,a/e,b/d,b/f,c/d,c/g,d/h,e/f,e/g,f/h,g/h}         
}                                                               
\end{tabular}

}
\caption{Two graphs with all vertices redundant, none surrounded.
The triangular prism is the smallest such graph. The cube is the
smallest such graph with diameter above 2.}
\label{f.rns}
\end{figure}

An interesting subclass of strong graphs is those that have no
surrounded vertices. Figure \ref{f.rns} shows the smallest two
examples. Table \ref{t.rs} lists some named graphs which have
this property. 

For 8 vertices there are 21 graphs with all vertices redundant
and none surrounded. All but two of them have diameter 2.
Those with diameter 3 are the cube and the cube plus two
parallel face-diagonals.

\begin{table}
\begin{tabular}{lcc}
   graph & vertices & diameter \\
\hline
triangular prism & 6 & 2 \\
$K_3 \otimes K_3$ & 9 & 2 \\
Shrikhande & 16 & 2 \\
Clebsch (folded 5-cube)   & 16       & 2 \\
Kneser($n,2$) for $n > 5$ & $^n C_2$ & 2 \\
Kneser($n,3$) for $n > 8$ & $^n C_3$ & 2 \\
crown $n > 3$             & $2n$     & 3 \\
icosahedron               & 12       & 3 \\
distance-3 graph of Heawood  & 14 & 3 \\
Gosset         & 56 & 3 \\
Hadamard       & 32 &  4 \\
hexacode       & 36 & 4 \\
Hadamard       & 48 & 4 \\ 
Suetake        & 72 & 4 \\
co-cliques in Hoffman-Singleton\!\!\!&\!\!\! 100 & 4 \\
doubled Gewirtz  & 112 & 5 \\
many Johnson graphs \\
$k$-cube & $2^k$ & $k$ 
\end{tabular}
\caption{Example graphs with all vertices redundant and none surrounded.}
\label{t.rs}
\end{table}

The $k-$cube graph has all vertices redundant and not surrounded. It has
$2^k$ vertices and diameter $k$.

\begin{conjecture}
The maximum diameter of a graph of order $n$ 
with all vertices redundant and none
surrounded is $\lfloor \log_2(n) \rfloor$; for each $n$ equal to a power
of 2 this maximum is realised only by the $k$-cube and graphs of
which the $k$-cube is a subgraph. 
\end{conjecture}

From a practical point of view, strong graphs with no surrounded
vertices realise a judicious compromise between redundancy and
efficiency. In applications such as communications networks, 
a little redundancy may be beneficial so that communications do not
break down entirely when one node is occupied or not available, but 
if the edges are themselves costly it may be 
inefficient to allow one node to surround another. The $k$-cube
is already a recognised solution that has been adopted in 
networks of processors for parallel computing. 

\subsection{Weak graphs (a.k.a. distance-critical graphs)}

The graphs we are calling `weakly linked' are the very ones previously
termed {\em distance-critical} by \Erdos\ and Howorka.\cite{Erdos1980}
The term {\em distance-critical} is, I think, more informative for
general usage in the literature, but I shall employ `weak' here
for the convenience of brevity.

The weak graphs begin at order $n=5$ where there is one weak graph.
Those of order in the range 5--8 are shown in Figs \ref{f.weak7},
\ref{f.weak8}.

\begin{lemma} \label{l.cycles}
A graph is weak iff
every vertex is part of a chordless cycle of length greater
than~4.
\end{lemma}

\begin{proof}
There are no endpoints in a weak graph so every vertex is in one or more
cycles. For any given vertex $v$, a triangle or a square of which it is
a part both imply a strong link between the relevant pair of
neighbours of $v$. But in order for $v$ to be non-redundant there must
be a pair of neighbours which are not strongly linked. Therefore $v$ is
part of a chordless cycle longer than 4. The reverse implication also
follows. 
\qed
\end{proof}

\begin{corollary}
The graphs with no endpoints, no isolated vertices and no
cycles of length less than 5 are weak. 
For example, 
 bi-connected graphs of girth above 4 are weak.
\end{corollary}

\begin{figure*}
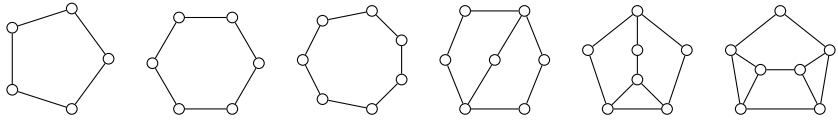

\centering{
\begin{tabular}{cccccc}
\tikz[scale=0.2]{\vertices{{( 3.5, 0)/a},{( 1.08155948, 3.32869781)/b},
{( -2.83155948, 2.05724838)/c},{( -2.83155948, -2.05724838)/d},        
{( 1.08155948, -3.32869781)/e}}                                        
\edges{a/b,a/e,b/c,c/d,d/e}                                            
}                                                                      
 & 
\tikz[scale=0.2]{\vertices{{( 3.5, 0)/a},{( 1.75, 3.03108891)/b},             
{( -1.75, 3.03108891)/c},{( -3.5, 4.2862638e-16)/d},{( -1.75, -3.03108891)/e},
{( 1.75, -3.03108891)/f}}                                                     
\edges{a/b,a/f,b/c,c/d,d/e,e/f}                                               
}                                                                             
 & 
\tikz[scale=0.2]{\vertices{{( 7.15, 2.6)/a},{( 1.95, 1.3)/b},       
{( 1.95, 6.5)/c},{( 5.2, 7.15)/d},{( 0.65, 3.9)/e},{( 5.2, 0.65)/f},
{( 7.15, 5.2)/g}}                                                   
\edges{a/f,a/g,b/e,b/f,c/d,c/e,d/g}                                 
}                                                                   
 & 
\tikz[scale=0.2]{\vertices{{( 5.85, 0.65)/a},{( 1.95, 0.65)/b},     
{( 1.95, 7.15)/c},{( 0.65, 3.9)/d},{( 3.9, 3.9)/e},{( 7.15, 3.9)/f},
{( 5.85, 7.15)/g}}                                                  
\edges{a/b,a/f,b/d,b/e,c/d,c/g,e/g,f/g}                             
}                                                                   
 & 
\tikz[scale=0.2]{\vertices{{( 5.85, 0.65)/a},{( 3.9, 2.6)/b},         
{( 3.9, 7.15)/c},{( 0.65, 4.55)/d},{( 7.15, 4.55)/e},{( 3.9, 4.55)/f},
{( 1.95, 0.65)/g}}                                                    
\edges{a/b,a/e,a/g,b/f,b/g,c/d,c/e,c/f,d/g}                           
}                                                                     
 & 
\tikz[scale=0.2]{\vertices{{( 1.3, 0.65)/a},{( 7.15, 4.55)/b},       
{( 0.65, 4.55)/c},{( 5.2, 3.25)/d},{( 3.9, 7.15)/e},{( 6.5, 0.65)/f},
{( 2.6, 3.25)/g}}                                                    
\edges{a/c,a/f,a/g,b/d,b/e,b/f,c/e,c/g,d/f,d/g}                      
}                                                                    
\end{tabular}

}
\caption{The weak graphs on up to 7 vertices.}
\label{f.weak7}
\end{figure*}

\begin{figure*}
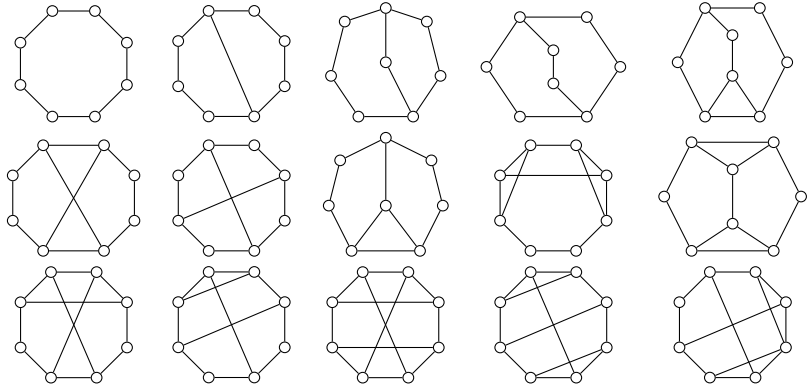

\centering{
\begin{tabular}{ccccc}
\tikz[scale=0.2]{\vertices{{( 7.7, 5.6)/a},{( 7.7, 2.8)/b},{( 2.8, 7.7)/c},
{( 5.6, 7.7)/d},{( 0.7, 5.6)/e},{( 5.6, 0.7)/f},{( 0.7, 2.8)/g},           
{( 2.8, 0.7)/h}}                                                           
\edges{a/b,a/d,b/f,c/d,c/e,e/g,f/h,g/h}                                    
}                                                                          
 & 
\tikz[scale=0.2]{\vertices{{( 6, 9)/a},{( 3, 2)/b},{( 1, 7)/c},
{( 3, 9)/d},{( 8, 4)/e},{( 6, 2)/f},{( 1, 4)/g},{( 8, 7)/h}}   
\edges{a/d,a/h,b/f,b/g,c/d,c/g,d/f,e/f,e/h}                    
}                                                              
 & 
\tikz[scale=0.2]{\vertices{{( 7.2, 1.8)/a},{( 9, 4.5)/b},{( 1.8, 4.5)/c},
{( 3.6, 1.8)/d},{( 2.7, 8.1)/e},{( 8.1, 8.1)/f},{( 5.4, 9)/g},           
{( 5.4, 5.4)/h}}                                                         
\edges{a/b,a/d,a/h,b/f,c/d,c/e,e/g,f/g,g/h}                              
}                                                                        
 & 
\tikz[scale=0.2]{\vertices{{( 8.8, 3.3)/a},{( 4.4, 3.3)/b},{( 6.6, 7.7)/c},
{( 8.8, 9.9)/d},{( 6.6, 5.5)/e},{( 11, 6.6)/f},{( 2.2, 6.6)/g},            
{( 4.4, 9.9)/h}}                                                           
\edges{a/b,a/e,a/f,b/g,c/e,c/h,d/f,d/h,g/h}                                
}                                                                          
 & 
\tikz[scale=0.2]{\vertices{{( 7.2, 9)/a},{( 3.6, 1.8)/b},{( 9, 5.4)/c},
{( 1.8, 5.4)/d},{( 5.4, 7.2)/e},{( 7.2, 1.8)/f},{( 3.6, 9)/g},         
{( 5.4, 4.5)/h}}                                                       
\edges{a/c,a/g,b/d,b/f,b/h,c/f,d/g,e/g,e/h,f/h}                        
}                                                                      
\\ 
\tikz[scale=0.2]{\vertices{{( 9, 7)/a},{( 7, 2)/b},{( 1, 7)/c},
{( 3, 9)/d},{( 9, 4)/e},{( 3, 2)/f},{( 1, 4)/g},{( 7, 9)/h}}   
\edges{a/e,a/h,b/d,b/e,b/f,c/d,c/g,d/h,f/g,f/h}                
}                                                              
 & 
\tikz[scale=0.2]{\vertices{{( 4, 2)/a},{( 9, 7)/b},{( 2, 7)/c},
{( 4, 9)/d},{( 7, 2)/e},{( 7, 9)/f},{( 2, 4)/g},{( 9, 4)/h}}   
\edges{a/e,a/g,b/f,b/g,b/h,c/d,c/g,d/e,d/f,e/h}                
}                                                              
 & 
\tikz[scale=0.2]{\vertices{{( 7.5, 3.75)/a},{( 6, 0.75)/b},{( 0.75, 6.75)/c},
{( 3.75, 8.25)/d},{( 1.5, 0.75)/e},{( 0, 3.75)/f},{( 3.75, 3.75)/g},         
{( 6.75, 6.75)/h}}                                                           
\edges{a/b,a/h,b/e,b/g,c/d,c/f,d/g,d/h,e/f,e/g}                              
}                                                                            
 & 
\tikz[scale=0.2]{\vertices{{( 1, 4)/a},{( 8, 4)/b},{( 3, 9)/c},
{( 6, 9)/d},{( 1, 7)/e},{( 3, 2)/f},{( 8, 7)/g},{( 6, 2)/h}}   
\edges{a/c,a/e,a/f,b/d,b/g,b/h,c/d,c/e,d/g,e/g,f/h}            
}                                                              
 & 
\tikz[scale=0.2]{\vertices{{( 9, 0.9)/a},{( 6.3, 6.3)/b},{( 6.3, 2.7)/c},
{( 3.6, 8.1)/d},{( 9, 8.1)/e},{( 1.8, 4.5)/f},{( 10.8, 4.5)/g},          
{( 3.6, 0.9)/h}}                                                         
\edges{a/c,a/g,a/h,b/c,b/d,b/e,c/h,d/e,d/f,e/g,f/h}                      
}                                                                        
\\ 
\tikz[scale=0.2]{\vertices{{( 4, 1)/a},{( 2, 3)/b},{( 4, 8)/c},
{( 7, 8)/d},{( 9, 3)/e},{( 7, 1)/f},{( 2, 6)/g},{( 9, 6)/h}}   
\edges{a/b,a/d,a/f,b/g,c/d,c/f,c/g,d/h,e/f,e/h,g/h}            
}                                                              
 & 
\tikz[scale=0.2]{\vertices{{( 9, 8)/a},{( 4, 10)/b},{( 2, 8)/c},
{( 7, 3)/d},{( 2, 5)/e},{( 9, 5)/f},{( 4, 3)/g},{( 7, 10)/h}}   
                                                                
\edges{a/e,a/f,a/h,b/c,b/d,b/h,c/e,c/h,d/f,d/g,e/g}             
}                                                               
 & 
\tikz[scale=0.2]{\vertices{{( 9, 1)/a},{( 11, 6)/b},{( 4, 6)/c},
{( 9, 8)/d},{( 6, 8)/e},{( 4, 3)/f},{( 11, 3)/g},{( 6, 1)/h}}   
                                                                
\edges{a/e,a/g,a/h,b/c,b/d,b/g,c/e,c/f,d/e,d/h,f/g,f/h}         
}                                                               
 & 
\tikz[scale=0.2]{\vertices{{( 8, 1)/a},{( 3, 6)/b},{( 8, 8)/c},
{( 3, 3)/d},{( 10, 3)/e},{( 5, 1)/f},{( 5, 8)/g},{( 10, 6)/h}} 
                                                               
\edges{a/e,a/f,a/g,b/c,b/d,b/g,c/g,c/h,d/f,d/h,e/f,e/h}        
}                                                              
 & 
\tikz[scale=0.2]{\vertices{{( 9, 4)/a},{( 2, 7)/b},{( 7, 2)/c},
{( 4, 2)/d},{( 9, 7)/e},{( 7, 9)/f},{( 2, 4)/g},{( 4, 9)/h}}   
\edges{a/c,a/d,a/e,a/f,b/g,b/h,c/d,c/h,d/g,e/f,e/g,f/h}        
}                                                              
\\ 
\end{tabular}

}
\caption{The weak graphs on 8 vertices. The first 5 are examples of
the construction indicated in lemma \ref{l.weaksub}. The 13th is a cube with two
edges moved. }
\label{f.weak8}
\end{figure*}

\begin{lemma} \label{l.weaksub}
Every graph of order $n$ with no isolated vertices and $k$
redundant vertices
is an induced subgraph of some weak graph of order $\le n+k+1$.
\end{lemma}

\begin{proof}
We present a construction to make a weak graph starting from any graph
$G$ having no isolated vertices.
If in the given graph $G$ there are $k$ redundant vertices then 
attach to it a $k$-legged `spider' as follows. First
form the disjoint union of $G$ with a star graph on
$k+1$ vertices. Then attach each redundant vertex in the graph to
a pendant vertex in the star. The resulting graph is
weak because it introduces a weak link around every vertex which
had been redundant, and the introduced vertices also participate
in weak links. \qed
\end{proof}

\begin{figure}
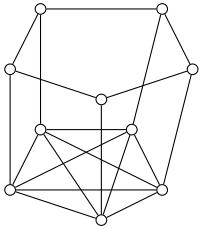

\centering{
\begin{tabular}{c}
\tikz[scale=0.2]{\vertices{{( 16, 8)/a},{( 12, 14)/b},{( 16, 20)/c},
{( 8, 12)/d},{( 6, 16)/e},{( 12, 6)/f},{( 18, 16)/g},{( 8, 20)/h},  
{( 14, 12)/i},{( 6, 8)/j}}                                          
\edges{a/d,a/f,a/g,a/i,a/j,b/e,b/f,b/g,c/g,c/h,c/i,d/f,d/h,d/i,     
d/j,e/h,e/j,f/i,f/j,i/j}                                            
}                                                                   
\end{tabular}

}
\caption{A weak graph which contains $K_5$ as a subgraph.
This is the only weak graph of order below 11
having this property.}
\label{f.K5}
\end{figure}

For many starting graphs there are constructions resulting in a
weak graph smaller than the one obtained by this `spider'
construction; 
there are examples in the figures. For example, if there are less
than three
redundant vertices it is always sufficient to add 
two further vertices to make the graph weak, and it may not be
necessary to add any new vertices (as, for example, when one converts
a path to a cycle). A further interesting example is shown in Fig.
\ref{f.K5} which shows the smallest weak graph containing the
complete graph $K_5$ as a subgraph.

The above constructions may reduce the graph diameter. A construction
which does not reduce the graph diameter is to introduce, for
every redundant vertex, a path graph P4 and attach its ends to
the redundant vertex, creating a chordless 5-vertex cycle. Alternatively,
introduce a pair of connected vertices (the graph P2) and attach it
to two non-adjacent neighbours of the non-redundant
vertex, thus creating a chordless 5-cycle.

\begin{theorem}
[\Erdos\ and Howorka]\cite{Erdos1980}  The number of edges in a weak graph of
order $n$ is bounded
from above by $|E| \le n(n-1)/2 - f(n)$ where
$f(n)$ is of order $\sqrt{2} n^{3/2}$.
\end{theorem}
The implication of this result is that a graph can have a lot of edges
while still remaining weak. Comparing $|E|$ with the number of vertex-pairs
in the graph, we have $|E| /\, ^n\!C_2 \simeq 1 - 2 \sqrt{2} n^{-1/2}$
for $n \gg 1$ when $|E|$ is at the upper bound, 
which indicates that most vertex-pairs can be adjacent for large $n$. 

\Erdos\ and Howorka express their limit in
the form $(1/\sqrt{2} + O(1))n^{3/2} \le f(n) \le (\sqrt{2} + O(1))n^{3/2}$
which amounts to the above. They add that they suspect
$f(k(k-5)/2 + k) = (k^3 - 6 k^2 + 7k)/2$. This formula gives $f(5) = 5$
which is correct, and $f(9) = 21$ which is incorrect because for
$n=9$ the maximum edge count is $18$, leading to $f(9) = 18$. This does not
deny that the formula may be accurate at many other values of $n$.

\subsubsection{Weak and $(k>1)$-order redundant}

The smallest graph with no $k$-th order redundant vertex is
of order $4k+1$ and is the cycle graph. The smallest graph with no
$k$'th-order redundant vertex and which is not a cycle graph is of
order $6k+1$. This can be seen by merging two cycle graphs.

For even $n$, a weak graph of order $n$ and diameter $n/2$ with all 
vertices 2nd-order redundant can be constructed as shown in Fig. \ref{f.link2}.

\begin{figure}
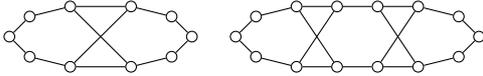

\centering{
\begin{tabular}{cc}
\tikz[scale=0.2]{\vertices{{( 13.3333333, 2.66666667)/a},{( 12, 4)/b},                     
{( 9.33333333, 4.66666667)/c},{( 5.33333333, 4.66666667)/d},{( 2.66666667, 4)/e},          
{( 1.33333333, 2.66666667)/f},{( 2.66666667, 1.33333333)/g},{( 5.33333333, 0.666666667)/h},
{( 9.33333333, 0.666666667)/i},{( 12, 1.33333333)/j}}                                      
\edges{a/b,a/j,b/c,c/d,c/h,d/e,d/i,e/f,f/g,g/h,h/i,i/j}                                    
}                                                                                          
 & 
\tikz[scale=0.2]{\vertices{{( 16, 4)/a},{( 14.6666667, 5.33333333)/b},  
{( 12, 6)/c},{( 9.33333333, 6)/d},{( 6.66666667, 6)/e},{( 4, 6)/f},     
{( 1.33333333, 5.33333333)/g},{( 0, 4)/h},{( 1.33333333, 2.66666667)/i},
{( 4, 2)/j},{( 6.66666667, 2)/k},{( 9.33333333, 2)/l},{( 12, 2)/m},     
{( 14.6666667, 2.66666667)/n}}                                          
\edges{a/b,a/n,b/c,c/d,c/l,d/e,d/m,e/f,e/j,f/g,f/k,g/h,h/i,i/j,         
j/k,k/l,l/m,m/n}                                                        
}                                                                       
\end{tabular}

}
\caption{The first two in a sequence of weak graphs with all vertices
2nd-order redundant.}
\label{f.link2}
\end{figure}

\section{Tensor products}

\begin{theorem}
	\begin{enumerate}
		\item The tensor product of two weak graphs is weak.
		\item The tensor product of two strong graphs without endpoints
		is strong.
	\end{enumerate}
\end{theorem}
\begin{proof}
	(i)
	Let the graphs be $A$ and $B$ and their tensor product be $T$. 
	To each vertex in $T$ there corresponds a pair of vertices $u \in A$, $v \in B$.
	By lemma \ref{l.cycles}, $u$ is a member of a chordless cycle in $A$ 
	of length $l(u) > 4$, and similarly for $v$ in $B$. Label the
	members of these cycles $A_1,A_2, \ldots, A_{l(u)}$ and
	$B_1,B_2, \ldots, B_{l(v)}$, respectively. In the tensor product graph there
	will be a cycle $(A_1,B_1),(A_2,B_2),\ldots$ where either vertex-list
	loops around when its cycle closes. The resulting cycle in $T$
	has a length equal to the least common multiple of $l(u)$ and $l(v)$
	and it is chordless. Hence the graph is weak by lemma \ref{l.cycles}.
	\qed \\
	(ii) 
	Let $u_1$-$u$-$u_2$ be a path in $A$ and let $v_1$-$v$-$v_2$ be a path in $B$.
	We consider the vertices $(u,v)$, $(u_1,v_1)$, $(u_2,v_2)$ in $T$.
	By construction the first of these is adjacent to the other two, and
	furthermore all pairs of neighbours of $(u,v)$ can be obtained this way.
	In order that $T$ be strong, it suffices if, for all such pairs of
	neighbours, either $(u_1,v_1)$ is
	adjacent to $(u_2,v_2)$, or both are adjacent to another vertex not
	equal to $(u,v)$. There are four possibilities:\\
	(1) $u_1 \in N(u_2)$ and $v_1 \in N(v_2)$;\\
	(2) $u_1 \notin N(u_2)$ and $v_1 \in N(v_2)$;\\
	(3) $u_1 \in N(u_2)$ and $v_1 \notin N(v_2)$;\\
	(4) $u_1 \notin N(u_2)$ and $v_1 \notin N(v_2)$.\\
	In case (1) $(u_1,v_1)$ is adjacent to $(u_2,v_2)$ in $T$. In cases (2)
	and (4) use the fact that $A$ is strong and therefore there is a vertex
	$u_3 \ne u$ which is adjacent to both $u_1$ and $u_2$. The vertex $(u_3,v)$
	is then adjacent to both $(u_1,v_1)$ and $(u_2,v_2)$ in $T$, therefore
	there is a strong link between the latter (using that $(u_3,v) \ne (u,v)$). 
	Finally case (3) is settled by the same construction employing a
	further vertex $v_3$ in $B$. \qed
\end{proof}

By combining this theorem with theorem \ref{th.isosurround} we deduce that
tensor products of any of the graphs in table \ref{t.rs} yield further
strong graphs with no surrounded vertices.

\section{Random graphs} \label{s.gnp}

We discuss random graphs in the $G(n,p)$ model proposed by
Gilbert and often named after \Erdos\ and R\'{e}nyi, in 
which each possible edge in a graph of order $n$
is assigned independently of the others with some fixed probability~$p$.\cite{Gilbert1959}

\begin{figure}
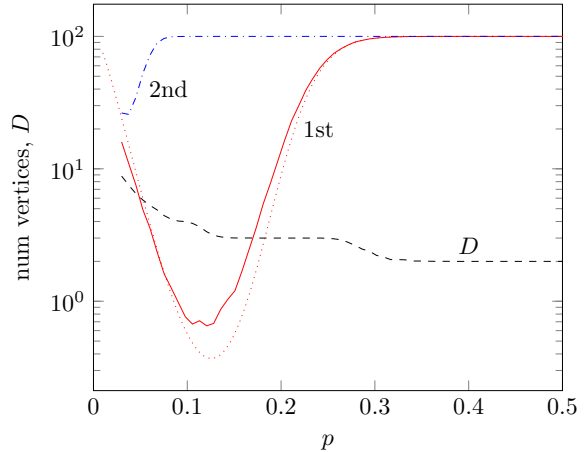

\mytikz{0.9}{Gnp100}
\caption{Properties of $G(n,p)$ graphs as a function of $p$ for $n=100$,
obtained by numerically generating a large number of graphs. Only
connected graphs are included.
Full curve: average number of redundant vertices; dashed: average diameter;
dash-dot: average number of 2nd-order-redundant vertices; dotted: 
$n P_{\rm red}$ as estimated using eqn (\ref{Pred}).}
\label{f.Gnp100}
\end{figure}

\begin{figure*}
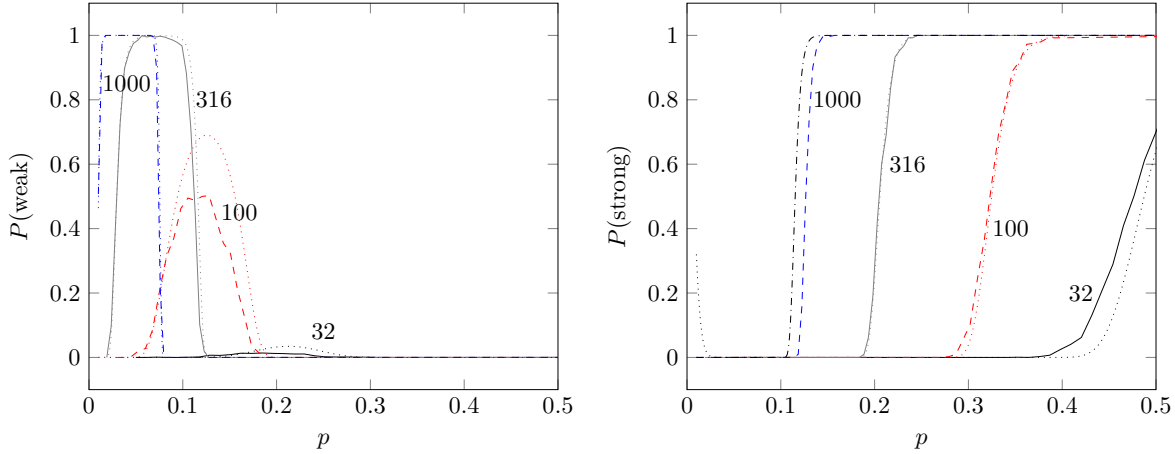

\mytikz{0.9}{Gnpweak}
\mytikz{0.9}{Gnpstrong}
\caption{Probability of weak (left) and strong (right) $G(n,p)$ graphs
as a function of $p$ for $n=32,100,316,1000$. Full and dashed curves: numerical;
dotted curves: eqn (\ref{Pred}). The dash-dot curve shows the probability
that the diameter is less than 3 (eqn (\ref{plowD})) for $n=1000$.}
\label{f.Gnpweak}
\end{figure*}

Figure \ref{f.Gnp100} shows the results of a numerical study
of such graphs on 100 vertices (i.e. $G(100,p)$). For values of $p$ in
the range 0 to 1 a large number of graphs was generated (using a pseudo-random
number generator). 
Graphs of this type almost always form a single connected component once 
$p$ is large enough.
For $n=100$ the proportion that are connected approaches 1 at $p \simeq 0.1$,
and approximately half are connected at $p=0.05$. 
For those that were connected, the diameter, the number of
redundant vertices, and other information was obtained. This information
is plotted in the figures. 

We observe that for these graphs of average diameter near 7 (at $p=0.04$) the
number of redundant vertices is on average around 10, which is 10\% of the
total number of vertices. Then as $p$ increases the average 
number of redundant vertices falls, falling below 1 as the diameter reaches 4.
The number then rises again as the diameter reaches 3 and tends to 100 (i.e. all
the vertices) as the diameter falls to 2. 
Note also that the number of 2nd-order-redundant vertices is quite high 
at all diameters.

Figure \ref{f.Gnpweak} shows
the proportion of these graphs that are weak and strong, as a function of $p$. 

The above features can be calculated to reasonable approximation as follows.
Let
\be
B^n_k(p) := \; ^n\!C_k \, p^k (1-p)^{n-k} 
\ee
where $^n\!C_k$ is a binomial coefficient. $B^n_k(p)$ is the binomial probability
distribution function. We would like to estimate the probability that any
given vertex in a $G(n,p)$ graph is redundant. Consider a vertex $v$ of degree $d$.
It has $^dC_2 = d(d-1)/2$ pairs of neighbours. We require that each pair of neighbours is either adjacent (forming a triangle with $v$) or both adjacent to some
other vertex (forming a square with $v$). The probability for this is
$$
\rule{-3ex}{0pt}P_{\rm red} =  B_0^{n-1}(p) + B_1^{n-1}(p) + \sum_{d=2}^{n-1} 
B_d^{n-1}(p) P_{\rm red}^{(d)}
$$
where $P_{\rm red}^{(d)}$ is given approximately by
\be
\rule{-3ex}{0pt}P_{\rm red}^{(d)} \simeq \left[ p + (1-p)(1 - B_0^{n-3}(p^2)) \right]^{^d C_2} .  \rule{-3ex}{0pt}     \label{Pred}
\ee
This expression is not exact because the different ways to
form squares are not entirely independent. For example, consider five
distinct vertices $s,t,u,v,w$ with $s,t,u \in N(v)$. The probability to have
both the edges $s$-$w$ and $t$-$w$ is $p^2$ and this is the probability to form
the square $\{s,v,t,w\}$ (given that $s,t \in N(v)$). Similarly the probability
to form the square $\{t,v,u,w\}$ is $p^2$.
But the probability to form both these squares is not $p^4$ but $p^3$. This
is not accounted for in (\ref{Pred}), with the result that the expression
underestimates
$P_{\rm red}$ by about a factor 2 when $P_{\rm red}$ is near its
minimum value. Meanwhile it gives $1-P_{\rm red}$ with
high (relative) accuracy when $P_{\rm red} \ll 1$ and therefore it reliably
indicates the range of $p$ values where $(1 - P_{\rm red})^n \simeq 1$
(so the graphs are weak) as well as the range where $(P_{\rm red})^n \simeq 1$
(so the graphs are strong).


\begin{figure}
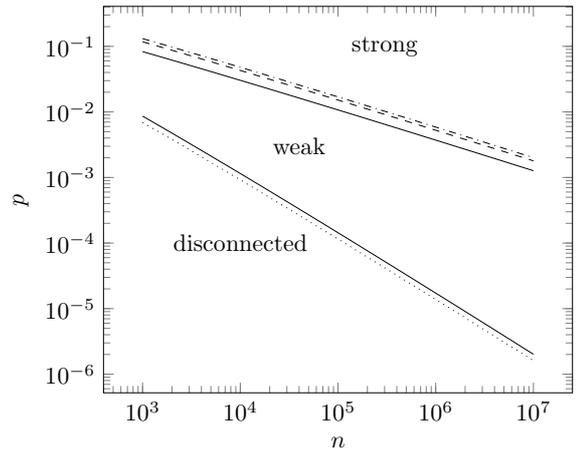

\mytikz{0.9}{Gnpthreshold}
\caption{Thresholds for various behaviours of $G(n,p)$ graphs.
Let $L(n) = \log(n) / n$. Then:
dotted line: $L(n)$ (graph is connected); lower full line: $(5/4)L(n)$ (graph is weak); upper full line: $(L(n))^{1/2}$ (graph no longer weak);
dashed: $(2L(n))^{1/2}$ (diameter is 2); dash-dot: $((5/2)L(n))^{1/2}$ (graph is strong).}
\label{f.Gnpregions}
\end{figure}

For large $n$ the $G(n,p)$ graphs show threshold-type behaviour
for various properties. For example, the probability that the diameter
is $\le 2$ is
\be
\left(1 - (1-p)( 1-p^2 )^{n-2} \right)^{ ^n\!C_2 }  \label{plowD}
\ee
(this is the probability that the vertices in every vertex pair are either 
adjacent to each other or both adjacent to some other vertex). This probability rises
steeply to 1 at a value of $p$ somewhat below the one where $P_{\rm strong}$
rises to 1 (Fig. \ref{f.Gnpweak}). Therefore in the limit of large $n$ the strong $G(n,p)$ graphs
have diameter 2. Nevertheless for $n=100$ and $p \simeq 0.25$ the probability
of obtaining a graph of diameter 3 with half the vertices redundant is 
high (about $0.97$). 

Some thresholds for various behaviours are indicated in Fig. \ref{f.Gnpregions}.

In view of the fact that both weak and strong graphs are very rare as a
fraction of all graphs of given order at high $n$, 
it is noteworthy that they are comparatively common 
among $G(n,p)$ graphs. Indeed for large $n$ there is a range of
$p$ where the $G(n,p)$ graphs are almost all weak---c.f. Figs
\ref{f.Gnpweak}, \ref{f.Gnpregions}. This is perhaps the most remarkable observation reported
in this paper.

\section{Generating strong and weak graphs}

We give brief comments on two tasks: that of combining strong or weak
graphs so as to give further examples, and that of exhaustive generation
of all strong or weak graphs of given order. 

\defin{The operation of {\em merging} two graphs $A$ and $B$ consists
in first finding an induced subgraph of $A$ which is isomorphic to
an induced subgraph of $B$, then 
identifying these subgraphs (that is, announcing that their 
vertices and edges in $A$ are none other than their partners in $B$), 
and then forming the union $(V_A \cup V_B, E_A \cup E_B)$.}

Examples are shown in Fig. \ref{f.merge}

\begin{figure*}
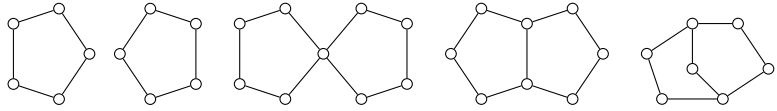

\centering{
\begin{tabular}{cccc}
\tikz[scale=0.2]{\vertices{{( 9, 2)/a},{( 12, 3)/b},{( 12, 7)/c},       
{( 9, 8)/d},{( 7, 5)/e},{( 0, 7)/f},{( 0, 3)/g},{( 3, 2)/h},{( 5, 5)/i},
{( 3, 8)/j}}                                                            
\edges{a/b,a/e,b/c,c/d,d/e,f/g,f/j,g/h,h/i,i/j}                         
}                                                                       
 & 
\tikz[scale=0.2]{\vertices{{( 11, 4)/a},{( 11, 8)/b},{( 8, 9)/c},
{( 5.5, 6)/d},{( 8, 3)/e},{( 0, 8)/f},{( 0, 4)/g},{( 3, 3)/h},   
{( 3, 9)/i}}                                                     
\edges{a/b,a/e,b/c,c/d,d/e,d/h,d/i,f/g,f/i,g/h}                  
}                                                                
 & 
\tikz[scale=0.2]{\vertices{{( 9, 3)/a},{( 11, 6)/b},{( 9, 9)/c},
{( 6, 8)/d},{( 6, 4)/e},{( 3, 9)/f},{( 1, 6)/g},{( 3, 3)/h}}    
\edges{a/b,a/e,b/c,c/d,d/e,d/f,e/h,f/g,g/h}                     
}                                                               
 & 
\tikz[scale=0.2]{\vertices{{( 11, 6)/a},{( 9, 9)/b},{( 6, 9)/c},
{( 6, 6)/d},{( 8, 4)/e},{( 3, 7)/f},{( 4, 4)/g}}                
\edges{a/b,a/e,b/c,c/d,c/f,d/e,e/g,f/g}                         
}                                                               
\end{tabular}

}
\caption{A sequence of weak graphs formed from the merger of
two C5 graphs.}
\label{f.merge}
\end{figure*}

In order that a graph obtained by the merger of two weak graphs
$A$ and $B$ is itself weak, it is 
sufficient if one (or both) of $A \setminus B$ or 
$B\setminus A$ does not offer weak links between vertices in the 
common subgraph.

In order that a graph obtained by the merger of two strong graphs
$A$ and $B$ is itself strong, it is 
sufficient if for each vertex $v$ in the common subgraph, either
$N(v) \subset A$ or $N(v) \subset B$ (or both).

\defin{The {\em short} operation is denoted $*$ and is as follows.
For a given graph $G$ containing vertex $v$, the graph $G * v$
is the one obtained by first adding edges between all the vertices
in $N(v)$ which were not already adjacent (creating a clique) and
then deleting $v$.}

\theorm{\label{th.short}
If a graph $G$ has all vertices redundant then 
$G * v$ has all vertices redundant, $\forall v \in G$.}

\begin{proof}
By lemma \ref{l.strong} the graph $G$ has no weak links. 
If $s$ is the vertex which will be shorted, then we prove that
$G * s$ has no weak links. Let $N=N(s)$ be the open neighbourhood of $s$ in $G$. The pairs of vertices in $G*s$ fall into 3 classes:  
either both, one or neither of the pair are in $N$. If both then those vertices are adjacent in $G * s$ and therefore strongly linked. If neither then no new 2-walk is introduced between them
nor is any existing 2-walk between them removed by shorting $s$,
hence they remain either strongly linked or separated by distance greater
than 2. Finally, 
consider a pair $\{u,v\}$ with $u \in N$, $v \notin N$.The interesting
case is where $v$ is adjacent to $u$. In this case the 
{\em short} operation introduces 2-walks from $v$ via $u$ to
any vertex in $N$ which was not previously adjacent to $u$. The question
then
arises whether there is thus introduced a weak link such that $u$
is no longer redundant. However if $v$ is adjacent to one member of $N$
then it must be adjacent to another by lemma \ref{str_simple}. It
follows that after the short there will be a strong link 
from $v$ to each member of $N$ so $u$ is not redundant.\qed
\end{proof}


One use of this theorem is to show that there exists a method to
exhaustively generate all strong graphs of order $n+1$ from those of order
$n$ without requiring the generation of all graphs of order $n+1$.
For each strong graph of order $n$ one performs an operation which
might be called `un-shorting'. That is, 
attach a vertex of degree $k$ to all members of
a $k$-clique, then remove zero or more edges among that $k$-clique.
To be sure of generating all possible strong graphs of order $n+1$
one must try this operation for a sufficient selection of the cliques
of the starting graph, but it is not necessary to try them all, because
the theorem allows a short of any vertex. One may imagine that to 
produce the smaller graph, a vertex was selected for shorting in 
the larger graph
with some specific set of vertex-invariant properties, such as least
degree and least number of triangles, etc. Therefore in order to reverse
such an operation it is sufficient to arrange that the introduced vertex
shall have those same properties. 

We have investigated a similar method for weak graphs, as follows.

\defin{The {\em partial short} operation is like the {\em short} operation
	but fewer edges are introduced. A partial short of
vertex $v$ consists
in introducing among the neighbours of $v$ the least number of
edges that suffice to raise the degrees of those vertices above 2,
and then deleting $v$. (In other words, the introduced edges just
suffice to leave no pendants remaining after the operation). The
specific set of edges to be introduced is left otherwise unspecified.}

\begin{conjecture}
The following operation on a weak graph will yield a weak
graph on fewer vertices. Identify the vertices with the least
number of degree-2 neighbours, and among those the ones of least
degree. Partially short one such vertex. Having done so, judiciously
remove further edges until the graph is weak.
\end{conjecture}

(The conjecture was tested for small graphs.)

\section{Enumeration}

The numbers of graphs of order up to 10 with various properties
related to distance-redundancy are shown in Tables \ref{t.stat} and
\ref{t.statcon}.


\begin{table}
\begin{tabular}{c|rrrrrr}
$n$ & no s. & all s. & no r. & all r. & \!\!r $\neg$ s & r $\neg$ s \\ 
    &       &        &       &        &            & \!\!$D\!>\!2$ \\
\hline
1 & 1& 0& 0& 1& 1 & 0\\
2 & 0& 2& 0& 2& 0 & 0\\
3 & 0& 3& 0& 3& 0 & 0\\
4 & 0& 7& 0& 7& 0 & 0\\
5 & 1& 14& 1& 15& 0 & 0\\
6 & 2& 41& 1& 50& 1 & 0\\
7 & 8& 121& 4& 202& 3 & 0\\
8 & 68& 499& 15& 1509& 21 & 2 \\
9 & 1338& 2644 &  168 & 22889 & 311 & 5  \\
10& \!\!77738& \!\!22571 & \!\!2253 & \!\!833279 & 15047 & 494 \\
\end{tabular}
\caption{Statistics of small graphs. The columns indicate the
number of graphs of order $n$ with, respectively: 
no surrounded vertices, all surrounded
vertices, no redundant, all redundant, all redundant and not surrounded,
all redundant not surrounded and graph diameter above 2.}
\label{t.stat}
\end{table}

\begin{table}
\begin{tabular}{c|rrrrrr}
$n$ & no s. & all s. & no r. & all r. & \!\!r $\neg$ s & r $\neg$ s \\ 
    &       &        &       &        &            & \!\!$D\!>\!2$ \\
\hline
 1& 0& 0& 0& 0& 0& 0 \\        
 2& 0& 0& 0& 0& 0& 0 \\        
 3& 0& 1& 0& 1& 0& 0 \\        
 4& 0& 3& 0& 3& 0& 0 \\        
 5& 1& 6& 1& 7& 0& 0 \\        
 6& 2& 22& 1& 30& 1& 0 \\      
 7& 8& 70& 4& 141& 3& 0 \\     
 8& 68& 339& 15& 1259& 21& 2 \\
 9& 1338 & 2024 & 168& 21176& 311& 5 \\
 10& \!\!77737 & \!\!19389 &  \!\!2252 & \!\!808821 & \!\!15047  &  \!\!494
\end{tabular}
\caption{Statistics of small connected graphs. The columns are as for
table \ref{t.stat}.}
\label{t.statcon}
\end{table}

The weak graphs of order $n$ are a very small fraction of the set of all graphs 
of order $n$. The evidence for small $n$ suggests it is a fraction which falls exponentially with $n$. That this continues for large $n$ is suggested by the
fact that for each weak graph $G$ there are a large number
of further graphs with one or more redundant vertices which can be obtained
by adding edges to, or removing edges from, $G$. I have not attempted
to enumerate weak graphs for general $n$.

For strong graphs a large subset can be enumerated as follows.

First note that a graph $G$ of order $n-1$ will yield a strong
graph of order $n$ when a complete vertex (one of degree $n-1$) is added,
if and only if $G$ has diameter less than 3. (Proof: if
the diameter of $G$ is 3 or more, then the full vertex is not redundant in the graph of order $n$; if the diameter of $G$ is 1 or 2, then every vertex pair is strongly linked in the graph of order $n$).  
Next observe that any graph of order
$n-2$ will yield a strong graph of order $n$ by the method of lemma
\ref{l.strongsub}(i). More generally, for $k>1$ any graph of order $n-k$ will yield
a strong graph of order $n$ by the addition of $k$ weak twins, each
of which is adjacent to all the vertices of the starting graph. We will 
now count these strong graphs. We limit our
study to the addition of weak twins, not strong twins (this makes it
easier to avoid double-counting). 

Consider the complement of a graph which has $k$ weak twins of degree
$n-k$. In the complement these vertices will have degree $k-1$ and they
will all be adjacent to each other, therefore they will form the whole
of a connected component equal to the complete graph $K_k$. Hence the
number of graphs with $k$ weak twins of degree $n-k$ 
is equal to the number of graphs of order $n$
which have a component equal to $K_k$. 

Let $g_n$ be the total number of simple graphs of order $n$.

We want to count graphs having either a $K_1$ or a $K_2$ or a $K_3$
component, and so on, (so as to count those with the corresponding sets of weak 
twins). Consider first $K_1$ and $K_2$. At order $n$ there are $g_{n-1}$
graphs with at least one $K_1$ component, and $g_{n-2}$ with at least one
$K_2$ component, and $g_{n-3}$ with both a $K_1$ and a $K_2$. Hence the
number with either one or more $K_1$ or  one or more $K_2$ or both is $g_{n-1} + g_{n-2} - g_{n-3}$.

Proceeding now to $K_1$, $K_2$, $K_3$ we start with $g_{n-1}+g_{n-2}+g_{n-3}$
then subtract from this the $\{1,2\}$ intersection which is of size $g_{n-3}$,
and the $\{1,3\}$ and $\{2,3\}$ intersections, of size $g_{n-4},\;g_{n-5}$
respectively. At this point we have added the $\{1,2,3\}$ intersection once
and subtracted it three times so we need to add it back in, giving the total
$$
\left( g_{n-1}+g_{n-2}+g_{n-3} \right) 
- \left( g_{n-3} + g_{n-4} + g_{n-5} \right) + g_{n-6}.
$$
Proceeding similarly for any $k$, we have that the number of graphs
of order $n$ containing one or more connected components $K_k$,
where $k$ may be any value in the range $1 \le k \le n$, is
\be
\kappa_n = \sum_{z=1}^n   \sum_{|S|=z} (-1)^{z-1}   g_S 
\label{countK}
\ee
where $N = \{1,2,3, \ldots n\}$ is the set of positive integers up to $n$;
$S \subset N$ is a subset of $N$; the second sum is a sum over
such subsets of size $z$, and $g_S$ is a shorthand for 
$g_{n-m}$ where $m = \sum_{s \in S} s$ (the sum of the members of $S$). 
The values of $\kappa_n$ for small $n$ are shown in table \ref{t.Kk}.

$\kappa_n$ is not itself the number of strong graphs of the type
under investigation, but with an adjustment it is. 
In the case of a full vertex (one of degree $n-1$) the graph is strong
iff the other $n-1$ vertices induce a graph of diameter 1 or 2,
as already noted. So (\ref{countK}) will count strong graphs of the
type under investigation if we replace $g_{n-1}$ in the sum
by the number of
graphs of order $n-1$ with diameter $< 3$. Example values
are shown in the last row of table~\ref{t.Kk}.
As an example: we learn from this table that, of the 1259 connected strong
graphs of order 8, there are 524 which have either a full vertex,
or a set of $k$ weak twins of degree $n-k$ for one or more $k$ in
the range $1 < k < n$. (The total 525 recorded
in the table also includes the empty graph.) The last row
in the table shows that these account for about half the connected
strong graphs of order 8 with diameter 1 or 2.

\begin{table*}
\begin{tabular}{ccccccccccccccc}
$n$    & 1& 2& 3& 4& 5& 6& 7& 8& 9& 10& 11 \\       
\hline               
all    & 1& 2& 3& 6& 14& 44& 187& 1195& 13377& 286976& 12279669 \\
strong & 1& 2& 2& 4& 8& 25& 91& 525& 5186& 103827&  4391398 \\
$D < 3$& 1& 1& 1& 3& 7& 27& 122& 1064& 18423& 719783 
\end{tabular}
\caption{The row `all' counts all graphs having one or more $K_k$ components
(of whatever $k$), given by (\ref{countK}). The row `strong' counts strong graphs having one or more full vertices or a set or sets of weak twins
surrounding all the other vertices, given by~(\ref{countK}) with $g_{n-1}$
replaced as explained in the text.
The third row gives the number of connected strong graphs of diameter less than $3$. The empty graph of order $n$ contributes one graph to the count in the second row, but not the third row.}
\label{t.Kk}
\end{table*}

\section{Nibbling and mesh graphs}

Let us consider the possibility of deleting a redundant 
vertex (if there is one), and
then deleting a redundant vertex (if there is one) in the graph that results,
and so continuing until the graph is either empty or weak. We shall
call this operation {\em nibbling}. 

\begin{figure}
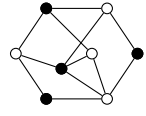

\centering{
\input nibble8.tex
}
\caption{An example of nibbling. The filled vertices are redundant. If one
first deletes the central redundant vertex the resulting graph is weak so
cannot be further reduced by nibbling. But if one first deletes one of
the other redundant vertices the resulting graph nibbles to zero.}
\label{f.nibble8}
\end{figure}

It can happen that a graph has more than one redundant vertex, such that
deleting one of them produces a weak graph (thus preventing further
nibbling) whereas deleting another does not. An example is shown in figure
\ref{f.nibble8}. Thus the question, whether or not any given graph can
be nibbled to zero, can depend on the choices made of which vertex to delete
at any stage.

Among the graphs that cannot be nibbled to zero are the weak graphs and
those obtained from a weak graph by the operation of adding
a simplicial vertex (that is, introducing a vertex and making it adjacent
to a clique in the weak graph) or splitting a vertex into twins. The set
of graphs that can be obtained this way
is a very small fraction of all graphs, which suggests that most
graphs can be nibbled to zero. That this is so is supported by some numerical
studies on small graphs, of which the results are shown in table \ref{t.nibble}.
In a test, 81\% of connected graphs on 10 vertices nibbled to
zero; 74\% of connected graphs on 10 vertices with average distance
above 2 nibbled to zero; 91\% of connected graphs on 10 vertices with average distance above 2 nibbled to zero or C5. Most $G(n,p)$ random graphs
can be reduced significantly by nibbling when 
$p$ is outside the range leading to a weak graph (c.f. Fig. \ref{f.Gnpweak}).

\begin{table}
\begin{tabular}{cccccc}
$n$ & NG$_n$ & rem & $d_{\rm av}>2$ & rem & C5 \\
8  & 11117 & 789   & 347            & 37  & 23 \\
9  &  261080 & 117720   & 4690      & 861 & 568 \\
10\!\!&\!\!11716571&\!\!\!\!2173100  & 89094     & 22916 & 14922
\end{tabular}
\caption{Statistics regarding nibbling. The first column shows the order,
the second shows the number of connected graphs, the third (`rem') the
number which did not nibble to zero. The fourth column shows the
number of connected graphs with average distance above 2, the fifth
the number of these which did not nibble to zero and the last column
shows how many of those nibbled to C5. The values in columns 3,5,6
may vary slightly depending on choices made
in the nibbling process.}
\label{t.nibble}
\end{table}

Consider now the regular {\em mesh graphs}, which we define as those with edges
and vertices corresponding to a regular tiling of
the Euclidean plane. We shall consider finite examples of such graphs,
that is, those that can be obtained by deleting vertices from an
infinite tiling such that a finite number remain. 

The tiling into hexagons yields a weak graph unless some vertex has all
neighbours, or all but one, deleted. We shall not discuss it further.
The tiling into squares or triangles yields a graph with redundant vertices
at corners where the boundary is convex. It is not 
hard to see that if the tiled region is simply
connected (in the topological sense) then the associated graph is guaranteed
to have redundant vertices and therefore will be reduced to zero under
nibbling operations. 

This leads to the following observation.

\begin{figure*}
	\centering{
		\begin{tabular}{ccc}
			\input meshA1.tex
			&
			\input meshA2.tex
			&
			\begin{tikzpicture}[scale= 0.25]
\draw( 1, 1.73205081)--( 1.25, 2.16506351);
\draw( 1, 1.73205081)--( 1.25, 1.29903811);
\draw( 1, 13.8564065)--( 1.25, 14.2894192);
\draw( 1, 13.8564065)--( 1.25, 13.4233938);
\draw( 1.5, 0.866025404)--( 2, 0.866025404);
\draw( 1.5, 0.866025404)--( 1.25, 1.29903811);
\draw( 1.5, 2.59807621)--( 2, 2.59807621);
\draw( 1.5, 2.59807621)--( 1.25, 2.16506351);
\draw( 2, 12.1243557)--( 2.25, 12.5573684);
\draw( 2, 12.1243557)--( 1.75, 12.5573684);
\draw( 2, 12.1243557)--( 2.25, 11.691343);
\draw( 1.5, 12.9903811)--( 2, 12.9903811);
\draw( 1.5, 12.9903811)--( 1.25, 13.4233938);
\draw( 1.5, 12.9903811)--( 1.75, 12.5573684);
\draw( 1.5, 14.7224319)--( 2, 14.7224319);
\draw( 1.5, 14.7224319)--( 1.25, 14.2894192);
\draw( 2.5, 0.866025404)--( 3, 0.866025404);
\draw( 2.5, 0.866025404)--( 2.75, 1.29903811);
\draw( 2.5, 0.866025404)--( 2, 0.866025404);
\draw( 3, 1.73205081)--( 2.75, 2.16506351);
\draw( 3, 1.73205081)--( 2.75, 1.29903811);
\draw( 3, 1.73205081)--( 3.25, 1.29903811);
\draw( 2.5, 2.59807621)--( 2, 2.59807621);
\draw( 2.5, 2.59807621)--( 2.75, 2.16506351);
\draw( 3, 10.3923048)--( 2.75, 10.8253175);
\draw( 3, 10.3923048)--( 3.25, 9.95929214);
\draw( 2.5, 11.2583302)--( 2.25, 11.691343);
\draw( 2.5, 11.2583302)--( 2.75, 10.8253175);
\draw( 2.5, 12.9903811)--( 2.75, 13.4233938);
\draw( 2.5, 12.9903811)--( 2, 12.9903811);
\draw( 2.5, 12.9903811)--( 2.25, 12.5573684);
\draw( 3, 13.8564065)--( 2.75, 14.2894192);
\draw( 3, 13.8564065)--( 2.75, 13.4233938);
\draw( 2.5, 14.7224319)--( 2, 14.7224319);
\draw( 2.5, 14.7224319)--( 2.75, 14.2894192);
\draw( 3.5, 0.866025404)--( 4, 0.866025404);
\draw( 3.5, 0.866025404)--( 3.25, 1.29903811);
\draw( 3.5, 0.866025404)--( 3, 0.866025404);
\draw( 4, 8.66025404)--( 4.5, 8.66025404);
\draw( 4, 8.66025404)--( 3.75, 9.09326674);
\draw( 3.5, 9.52627944)--( 3.25, 9.95929214);
\draw( 3.5, 9.52627944)--( 3.75, 9.09326674);
\draw( 4.5, 0.866025404)--( 5, 0.866025404);
\draw( 4.5, 0.866025404)--( 4, 0.866025404);
\draw( 5, 8.66025404)--( 5.25, 9.09326674);
\draw( 5, 8.66025404)--( 4.5, 8.66025404);
\draw( 5.5, 0.866025404)--( 6, 0.866025404);
\draw( 5.5, 0.866025404)--( 5, 0.866025404);
\draw( 5.5, 9.52627944)--( 5.75, 9.95929214);
\draw( 5.5, 9.52627944)--( 5.25, 9.09326674);
\draw( 6, 10.3923048)--( 6.25, 10.8253175);
\draw( 6, 10.3923048)--( 5.75, 9.95929214);
\draw( 6.5, 0.866025404)--( 7, 0.866025404);
\draw( 6.5, 0.866025404)--( 6, 0.866025404);
\draw( 6.5, 11.2583302)--( 6.75, 11.691343);
\draw( 6.5, 11.2583302)--( 6.25, 10.8253175);
\draw( 7, 12.1243557)--( 7.25, 12.5573684);
\draw( 7, 12.1243557)--( 6.75, 11.691343);
\draw( 7.5, 0.866025404)--( 8, 0.866025404);
\draw( 7.5, 0.866025404)--( 7, 0.866025404);
\draw( 7.5, 12.9903811)--( 8, 12.9903811);
\draw( 7.5, 12.9903811)--( 7.25, 12.5573684);
\draw( 8.5, 0.866025404)--( 9, 0.866025404);
\draw( 8.5, 0.866025404)--( 8, 0.866025404);
\draw( 9, 8.66025404)--( 9.5, 8.66025404);
\draw( 9, 8.66025404)--( 8.75, 9.09326674);
\draw( 8.5, 9.52627944)--( 8.75, 9.95929214);
\draw( 8.5, 9.52627944)--( 8.75, 9.09326674);
\draw( 9, 10.3923048)--( 9.5, 10.3923048);
\draw( 9, 10.3923048)--( 8.75, 9.95929214);
\draw( 8.5, 12.9903811)--( 9, 12.9903811);
\draw( 8.5, 12.9903811)--( 8, 12.9903811);
\draw( 9.5, 0.866025404)--( 10, 0.866025404);
\draw( 9.5, 0.866025404)--( 9, 0.866025404);
\draw( 10, 8.66025404)--( 10.5, 8.66025404);
\draw( 10, 8.66025404)--( 9.5, 8.66025404);
\draw( 10, 10.3923048)--( 10.25, 10.8253175);
\draw( 10, 10.3923048)--( 9.5, 10.3923048);
\draw( 9.5, 12.9903811)--( 10, 12.9903811);
\draw( 9.5, 12.9903811)--( 9, 12.9903811);
\draw( 10.5, 0.866025404)--( 11, 0.866025404);
\draw( 10.5, 0.866025404)--( 10, 0.866025404);
\draw( 11, 8.66025404)--( 11.5, 8.66025404);
\draw( 11, 8.66025404)--( 10.5, 8.66025404);
\draw( 10.5, 11.2583302)--( 11, 11.2583302);
\draw( 10.5, 11.2583302)--( 10.25, 10.8253175);
\draw( 10.5, 12.9903811)--( 11, 12.9903811);
\draw( 10.5, 12.9903811)--( 10, 12.9903811);
\draw( 11.5, 0.866025404)--( 12, 0.866025404);
\draw( 11.5, 0.866025404)--( 11, 0.866025404);
\draw( 12, 8.66025404)--( 11.5, 8.66025404);
\draw( 12, 8.66025404)--( 12.25, 8.22724134);
\draw( 11.5, 11.2583302)--( 12, 11.2583302);
\draw( 11.5, 11.2583302)--( 11, 11.2583302);
\draw( 11.5, 12.9903811)--( 12, 12.9903811);
\draw( 11.5, 12.9903811)--( 11, 12.9903811);
\draw( 12.5, 0.866025404)--( 12.75, 1.29903811);
\draw( 12.5, 0.866025404)--( 12, 0.866025404);
\draw( 13, 1.73205081)--( 13.25, 2.16506351);
\draw( 13, 1.73205081)--( 12.75, 1.29903811);
\draw( 12.5, 7.79422863)--( 13, 7.79422863);
\draw( 12.5, 7.79422863)--( 12.25, 8.22724134);
\draw( 12.5, 11.2583302)--( 12.75, 11.691343);
\draw( 12.5, 11.2583302)--( 12, 11.2583302);
\draw( 13, 12.1243557)--( 12.75, 12.5573684);
\draw( 13, 12.1243557)--( 12.75, 11.691343);
\draw( 12.5, 12.9903811)--( 12, 12.9903811);
\draw( 12.5, 12.9903811)--( 12.75, 12.5573684);
\draw( 13.5, 2.59807621)--( 13.75, 3.03108891);
\draw( 13.5, 2.59807621)--( 13.25, 2.16506351);
\draw( 14, 3.46410162)--( 13.75, 3.89711432);
\draw( 14, 3.46410162)--( 13.75, 3.03108891);
\draw( 13.5, 4.33012702)--( 13.75, 4.76313972);
\draw( 13.5, 4.33012702)--( 13.75, 3.89711432);
\draw( 14, 5.19615242)--( 14.25, 5.62916512);
\draw( 14, 5.19615242)--( 13.75, 4.76313972);
\draw( 13.5, 7.79422863)--( 14, 7.79422863);
\draw( 13.5, 7.79422863)--( 13, 7.79422863);
\draw( 14.5, 6.06217783)--( 14.75, 6.49519053);
\draw( 14.5, 6.06217783)--( 14.25, 5.62916512);
\draw( 15, 6.92820323)--( 14.75, 7.36121593);
\draw( 15, 6.92820323)--( 14.75, 6.49519053);
\draw( 14.5, 7.79422863)--( 14, 7.79422863);
\draw( 14.5, 7.79422863)--( 14.75, 7.36121593);
\end{tikzpicture}
	\end{tabular}}
	\caption{An illustration of the nibbling operation applied to a mesh
		graph in a simply connected region. The left diagram shows the starting
		graph (after deleting two vertices). The middle diagram shows an intermediate
		result. The right diagram shows the final weak graph which is isometric with
		the starting graph.}
	\label{f.meshnibble1}
\end{figure*}
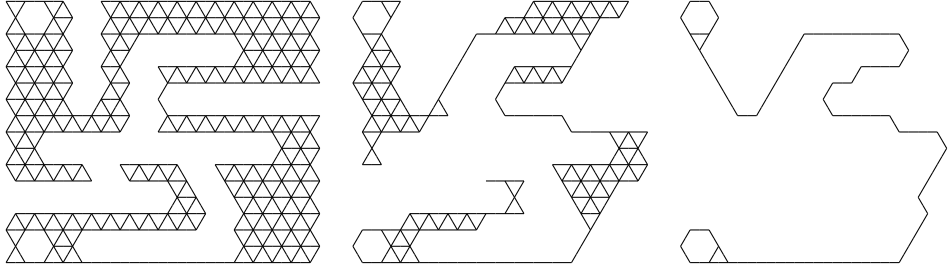

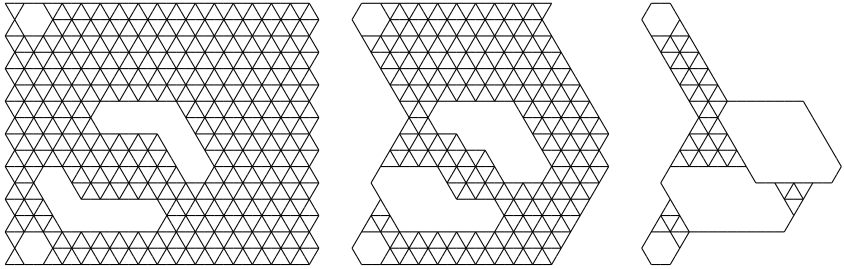
\begin{figure*}
	\centering{
		\begin{tabular}{ccc}
			\input meshB1.tex
			&
			\input meshB2.tex
			&
			\begin{tikzpicture}[scale= 0.25]
\draw( 1, 1.73205081)--( 1.25, 2.16506351);
\draw( 1, 1.73205081)--( 1.25, 1.29903811);
\draw( 1, 13.8564065)--( 1.25, 14.2894192);
\draw( 1, 13.8564065)--( 1.25, 13.4233938);
\draw( 1.5, 0.866025404)--( 2, 0.866025404);
\draw( 1.5, 0.866025404)--( 1.25, 1.29903811);
\draw( 1.5, 2.59807621)--( 2, 2.59807621);
\draw( 1.5, 2.59807621)--( 1.75, 3.03108891);
\draw( 1.5, 2.59807621)--( 1.25, 2.16506351);
\draw( 2, 3.46410162)--( 2.5, 3.46410162);
\draw( 2, 3.46410162)--( 2.25, 3.89711432);
\draw( 2, 3.46410162)--( 1.75, 3.03108891);
\draw( 2, 3.46410162)--( 2.25, 3.03108891);
\draw( 2, 5.19615242)--( 2.25, 5.62916512);
\draw( 2, 5.19615242)--( 2.25, 4.76313972);
\draw( 2, 12.1243557)--( 2.5, 12.1243557);
\draw( 2, 12.1243557)--( 2.25, 12.5573684);
\draw( 2, 12.1243557)--( 1.75, 12.5573684);
\draw( 2, 12.1243557)--( 2.25, 11.691343);
\draw( 1.5, 12.9903811)--( 2, 12.9903811);
\draw( 1.5, 12.9903811)--( 1.25, 13.4233938);
\draw( 1.5, 12.9903811)--( 1.75, 12.5573684);
\draw( 1.5, 14.7224319)--( 2, 14.7224319);
\draw( 1.5, 14.7224319)--( 1.25, 14.2894192);
\draw( 2.5, 0.866025404)--( 2.75, 1.29903811);
\draw( 2.5, 0.866025404)--( 2, 0.866025404);
\draw( 3, 1.73205081)--( 3.25, 2.16506351);
\draw( 3, 1.73205081)--( 2.75, 2.16506351);
\draw( 3, 1.73205081)--( 2.75, 1.29903811);
\draw( 2.5, 2.59807621)--( 3, 2.59807621);
\draw( 2.5, 2.59807621)--( 2.75, 3.03108891);
\draw( 2.5, 2.59807621)--( 2.25, 3.03108891);
\draw( 2.5, 2.59807621)--( 2, 2.59807621);
\draw( 2.5, 2.59807621)--( 2.75, 2.16506351);
\draw( 3, 3.46410162)--( 2.75, 3.89711432);
\draw( 3, 3.46410162)--( 2.5, 3.46410162);
\draw( 3, 3.46410162)--( 2.75, 3.03108891);
\draw( 3, 3.46410162)--( 3.25, 3.03108891);
\draw( 2.5, 4.33012702)--( 2.25, 4.76313972);
\draw( 2.5, 4.33012702)--( 2.25, 3.89711432);
\draw( 2.5, 4.33012702)--( 2.75, 3.89711432);
\draw( 2.5, 6.06217783)--( 3, 6.06217783);
\draw( 2.5, 6.06217783)--( 2.75, 6.49519053);
\draw( 2.5, 6.06217783)--( 2.25, 5.62916512);
\draw( 3, 6.92820323)--( 3.5, 6.92820323);
\draw( 3, 6.92820323)--( 3.25, 7.36121593);
\draw( 3, 6.92820323)--( 2.75, 6.49519053);
\draw( 3, 6.92820323)--( 3.25, 6.49519053);
\draw( 3, 10.3923048)--( 3.5, 10.3923048);
\draw( 3, 10.3923048)--( 3.25, 10.8253175);
\draw( 3, 10.3923048)--( 2.75, 10.8253175);
\draw( 3, 10.3923048)--( 3.25, 9.95929214);
\draw( 2.5, 11.2583302)--( 3, 11.2583302);
\draw( 2.5, 11.2583302)--( 2.75, 11.691343);
\draw( 2.5, 11.2583302)--( 2.25, 11.691343);
\draw( 2.5, 11.2583302)--( 2.75, 10.8253175);
\draw( 3, 12.1243557)--( 3.5, 12.1243557);
\draw( 3, 12.1243557)--( 3.25, 12.5573684);
\draw( 3, 12.1243557)--( 2.75, 12.5573684);
\draw( 3, 12.1243557)--( 2.5, 12.1243557);
\draw( 3, 12.1243557)--( 2.75, 11.691343);
\draw( 3, 12.1243557)--( 3.25, 11.691343);
\draw( 2.5, 12.9903811)--( 3, 12.9903811);
\draw( 2.5, 12.9903811)--( 2.75, 13.4233938);
\draw( 2.5, 12.9903811)--( 2, 12.9903811);
\draw( 2.5, 12.9903811)--( 2.25, 12.5573684);
\draw( 2.5, 12.9903811)--( 2.75, 12.5573684);
\draw( 3, 13.8564065)--( 2.75, 14.2894192);
\draw( 3, 13.8564065)--( 2.75, 13.4233938);
\draw( 3, 13.8564065)--( 3.25, 13.4233938);
\draw( 2.5, 14.7224319)--( 2, 14.7224319);
\draw( 2.5, 14.7224319)--( 2.75, 14.2894192);
\draw( 3.5, 2.59807621)--( 4, 2.59807621);
\draw( 3.5, 2.59807621)--( 3.25, 3.03108891);
\draw( 3.5, 2.59807621)--( 3, 2.59807621);
\draw( 3.5, 2.59807621)--( 3.25, 2.16506351);
\draw( 3.5, 6.06217783)--( 4, 6.06217783);
\draw( 3.5, 6.06217783)--( 3.75, 6.49519053);
\draw( 3.5, 6.06217783)--( 3.25, 6.49519053);
\draw( 3.5, 6.06217783)--( 3, 6.06217783);
\draw( 4, 6.92820323)--( 4.5, 6.92820323);
\draw( 4, 6.92820323)--( 4.25, 7.36121593);
\draw( 4, 6.92820323)--( 3.75, 7.36121593);
\draw( 4, 6.92820323)--( 3.5, 6.92820323);
\draw( 4, 6.92820323)--( 3.75, 6.49519053);
\draw( 4, 6.92820323)--( 4.25, 6.49519053);
\draw( 3.5, 7.79422863)--( 4, 7.79422863);
\draw( 3.5, 7.79422863)--( 3.75, 8.22724134);
\draw( 3.5, 7.79422863)--( 3.25, 7.36121593);
\draw( 3.5, 7.79422863)--( 3.75, 7.36121593);
\draw( 4, 8.66025404)--( 4.5, 8.66025404);
\draw( 4, 8.66025404)--( 4.25, 9.09326674);
\draw( 4, 8.66025404)--( 3.75, 9.09326674);
\draw( 4, 8.66025404)--( 3.75, 8.22724134);
\draw( 4, 8.66025404)--( 4.25, 8.22724134);
\draw( 3.5, 9.52627944)--( 4, 9.52627944);
\draw( 3.5, 9.52627944)--( 3.75, 9.95929214);
\draw( 3.5, 9.52627944)--( 3.25, 9.95929214);
\draw( 3.5, 9.52627944)--( 3.75, 9.09326674);
\draw( 4, 10.3923048)--( 4.5, 10.3923048);
\draw( 4, 10.3923048)--( 4.25, 10.8253175);
\draw( 4, 10.3923048)--( 3.75, 10.8253175);
\draw( 4, 10.3923048)--( 3.5, 10.3923048);
\draw( 4, 10.3923048)--( 3.75, 9.95929214);
\draw( 4, 10.3923048)--( 4.25, 9.95929214);
\draw( 3.5, 11.2583302)--( 4, 11.2583302);
\draw( 3.5, 11.2583302)--( 3.75, 11.691343);
\draw( 3.5, 11.2583302)--( 3.25, 11.691343);
\draw( 3.5, 11.2583302)--( 3, 11.2583302);
\draw( 3.5, 11.2583302)--( 3.25, 10.8253175);
\draw( 3.5, 11.2583302)--( 3.75, 10.8253175);
\draw( 4, 12.1243557)--( 3.75, 12.5573684);
\draw( 4, 12.1243557)--( 3.5, 12.1243557);
\draw( 4, 12.1243557)--( 3.75, 11.691343);
\draw( 4, 12.1243557)--( 4.25, 11.691343);
\draw( 3.5, 12.9903811)--( 3.25, 13.4233938);
\draw( 3.5, 12.9903811)--( 3, 12.9903811);
\draw( 3.5, 12.9903811)--( 3.25, 12.5573684);
\draw( 3.5, 12.9903811)--( 3.75, 12.5573684);
\draw( 4.5, 2.59807621)--( 5, 2.59807621);
\draw( 4.5, 2.59807621)--( 4, 2.59807621);
\draw( 4.5, 6.06217783)--( 5, 6.06217783);
\draw( 4.5, 6.06217783)--( 4.75, 6.49519053);
\draw( 4.5, 6.06217783)--( 4.25, 6.49519053);
\draw( 4.5, 6.06217783)--( 4, 6.06217783);
\draw( 5, 6.92820323)--( 5.5, 6.92820323);
\draw( 5, 6.92820323)--( 5.25, 7.36121593);
\draw( 5, 6.92820323)--( 4.75, 7.36121593);
\draw( 5, 6.92820323)--( 4.5, 6.92820323);
\draw( 5, 6.92820323)--( 4.75, 6.49519053);
\draw( 5, 6.92820323)--( 5.25, 6.49519053);
\draw( 4.5, 7.79422863)--( 5, 7.79422863);
\draw( 4.5, 7.79422863)--( 4.75, 8.22724134);
\draw( 4.5, 7.79422863)--( 4.25, 8.22724134);
\draw( 4.5, 7.79422863)--( 4, 7.79422863);
\draw( 4.5, 7.79422863)--( 4.25, 7.36121593);
\draw( 4.5, 7.79422863)--( 4.75, 7.36121593);
\draw( 5, 8.66025404)--( 5.25, 9.09326674);
\draw( 5, 8.66025404)--( 4.75, 9.09326674);
\draw( 5, 8.66025404)--( 4.5, 8.66025404);
\draw( 5, 8.66025404)--( 4.75, 8.22724134);
\draw( 5, 8.66025404)--( 5.25, 8.22724134);
\draw( 4.5, 9.52627944)--( 5, 9.52627944);
\draw( 4.5, 9.52627944)--( 4.75, 9.95929214);
\draw( 4.5, 9.52627944)--( 4.25, 9.95929214);
\draw( 4.5, 9.52627944)--( 4, 9.52627944);
\draw( 4.5, 9.52627944)--( 4.25, 9.09326674);
\draw( 4.5, 9.52627944)--( 4.75, 9.09326674);
\draw( 5, 10.3923048)--( 4.75, 10.8253175);
\draw( 5, 10.3923048)--( 4.5, 10.3923048);
\draw( 5, 10.3923048)--( 4.75, 9.95929214);
\draw( 5, 10.3923048)--( 5.25, 9.95929214);
\draw( 4.5, 11.2583302)--( 4.25, 11.691343);
\draw( 4.5, 11.2583302)--( 4, 11.2583302);
\draw( 4.5, 11.2583302)--( 4.25, 10.8253175);
\draw( 4.5, 11.2583302)--( 4.75, 10.8253175);
\draw( 5.5, 2.59807621)--( 6, 2.59807621);
\draw( 5.5, 2.59807621)--( 5, 2.59807621);
\draw( 5.5, 6.06217783)--( 6, 6.06217783);
\draw( 5.5, 6.06217783)--( 5.75, 6.49519053);
\draw( 5.5, 6.06217783)--( 5.25, 6.49519053);
\draw( 5.5, 6.06217783)--( 5, 6.06217783);
\draw( 6, 6.92820323)--( 5.75, 7.36121593);
\draw( 6, 6.92820323)--( 5.5, 6.92820323);
\draw( 6, 6.92820323)--( 5.75, 6.49519053);
\draw( 6, 6.92820323)--( 6.25, 6.49519053);
\draw( 5.5, 7.79422863)--( 5.25, 8.22724134);
\draw( 5.5, 7.79422863)--( 5, 7.79422863);
\draw( 5.5, 7.79422863)--( 5.25, 7.36121593);
\draw( 5.5, 7.79422863)--( 5.75, 7.36121593);
\draw( 5.5, 9.52627944)--( 6, 9.52627944);
\draw( 5.5, 9.52627944)--( 5.25, 9.95929214);
\draw( 5.5, 9.52627944)--( 5, 9.52627944);
\draw( 5.5, 9.52627944)--( 5.25, 9.09326674);
\draw( 6.5, 2.59807621)--( 7, 2.59807621);
\draw( 6.5, 2.59807621)--( 6, 2.59807621);
\draw( 7, 5.19615242)--( 7.5, 5.19615242);
\draw( 7, 5.19615242)--( 6.75, 5.62916512);
\draw( 6.5, 6.06217783)--( 6.25, 6.49519053);
\draw( 6.5, 6.06217783)--( 6, 6.06217783);
\draw( 6.5, 6.06217783)--( 6.75, 5.62916512);
\draw( 6.5, 9.52627944)--( 7, 9.52627944);
\draw( 6.5, 9.52627944)--( 6, 9.52627944);
\draw( 7.5, 2.59807621)--( 8, 2.59807621);
\draw( 7.5, 2.59807621)--( 7, 2.59807621);
\draw( 8, 5.19615242)--( 8.5, 5.19615242);
\draw( 8, 5.19615242)--( 7.5, 5.19615242);
\draw( 8, 5.19615242)--( 8.25, 4.76313972);
\draw( 7.5, 9.52627944)--( 8, 9.52627944);
\draw( 7.5, 9.52627944)--( 7, 9.52627944);
\draw( 8.5, 2.59807621)--( 8.75, 3.03108891);
\draw( 8.5, 2.59807621)--( 8, 2.59807621);
\draw( 9, 3.46410162)--( 9.25, 3.89711432);
\draw( 9, 3.46410162)--( 8.75, 3.89711432);
\draw( 9, 3.46410162)--( 8.75, 3.03108891);
\draw( 8.5, 4.33012702)--( 9, 4.33012702);
\draw( 8.5, 4.33012702)--( 8.75, 4.76313972);
\draw( 8.5, 4.33012702)--( 8.25, 4.76313972);
\draw( 8.5, 4.33012702)--( 8.75, 3.89711432);
\draw( 9, 5.19615242)--( 9.5, 5.19615242);
\draw( 9, 5.19615242)--( 8.5, 5.19615242);
\draw( 9, 5.19615242)--( 8.75, 4.76313972);
\draw( 9, 5.19615242)--( 9.25, 4.76313972);
\draw( 8.5, 9.52627944)--( 9, 9.52627944);
\draw( 8.5, 9.52627944)--( 8, 9.52627944);
\draw( 9.5, 4.33012702)--( 9.75, 4.76313972);
\draw( 9.5, 4.33012702)--( 9.25, 4.76313972);
\draw( 9.5, 4.33012702)--( 9, 4.33012702);
\draw( 9.5, 4.33012702)--( 9.25, 3.89711432);
\draw( 10, 5.19615242)--( 10.5, 5.19615242);
\draw( 10, 5.19615242)--( 9.5, 5.19615242);
\draw( 10, 5.19615242)--( 9.75, 4.76313972);
\draw( 10, 8.66025404)--( 9.75, 9.09326674);
\draw( 10, 8.66025404)--( 10.25, 8.22724134);
\draw( 9.5, 9.52627944)--( 9, 9.52627944);
\draw( 9.5, 9.52627944)--( 9.75, 9.09326674);
\draw( 11, 5.19615242)--( 11.25, 5.62916512);
\draw( 11, 5.19615242)--( 10.5, 5.19615242);
\draw( 11, 6.92820323)--( 10.75, 7.36121593);
\draw( 11, 6.92820323)--( 11.25, 6.49519053);
\draw( 10.5, 7.79422863)--( 10.25, 8.22724134);
\draw( 10.5, 7.79422863)--( 10.75, 7.36121593);
\draw( 11.5, 6.06217783)--( 11.25, 6.49519053);
\draw( 11.5, 6.06217783)--( 11.25, 5.62916512);
\end{tikzpicture}
	\end{tabular}}
	\caption{An illustration of the nibbling operation applied to a mesh
		graph in a non-simply-connected region. The left diagram shows the starting
		graph (after deleting two vertices). The middle diagram shows an intermediate
		result. The right diagram shows the final weak graph which is isometric with
		the starting graph.}
	\label{f.meshnibble2}
\end{figure*}

Suppose we have a simply-connected region which has been tiled with 
triangles (for example) and we should like to know the graph distance between
two chosen vertices which are not on the boundary. We proceed as follows.
First delete those two vertices, then nibble the graph until it is weak.
There will then remain a much smaller graph in which there is a path between the
6-cycles surrounding the deleted vertices, such that the graph distance
between them is the same as in the original graph. Figure \ref{f.meshnibble1}
shows an example. 

This same idea can be applied also to a non-simply-connected region,
but now the weak graph that remains after nibbling will not be reduced
as much as in the simpler case, so there still remains some 
computational work to do to find the
distance in that graph; c.f. Fig. \ref{f.meshnibble2}. 

By exploiting the simplicity and properties of mesh graphs one can find
redundant vertices in them with very little computational expense, with
the result that the nibbling approach may offer some practical computational
advantage in some settings. But whether or not it does so, it is a
pleasing observation which has something in common with the way the human
brain detects small paths between points in the visual field
by first eliminating from consideration regions which are irrelevant.

\section{Relation to distance-hereditary graphs}

\defin{A {\em distance-hereditary} graph is one in which the matrix of
distances in any connected induced subgraph is the same as the corresponding section of the distance matrix of the original graph.}

Distance-hereditary graphs have been much studied and much
is known about them. The definition implies that deletion of any set
of vertices which is not a cut-set is an isometric operation. This
implies that every vertex in a biconnected distance-hereditary graph
is redundant, and every vertex which is not a cut-vertex
in a singly connected (i.e. separable) distance-hereditary graph is redundant.

A distance-hereditary graph has at least two surrounded vertices.
Proof: the graph P2 has two surrounded vertices and is distance-hereditary.
All other distance-hereditary graphs can be formed by one or more operations of
adding a pendant or splitting a vertex into twins. Neither of those
operations can reduce the number of surrounded vertices. \qed

The set of strong graphs is intermediate between the set of all graphs and
the set of distance-hereditary graphs. The latter is not a subset of
the set of strong graphs because our definition of {\em strong} does
not allow pendants for order $n > 2$. 
Also the property of being strong is far from
guaranteeing isometric induced subgraphs in general. But if all
subgraphs obtained by deleting a single vertex from a strong graph
are themselves strong, and all those by deleting a further vertex, and
so on, then the original graph is distance-hereditary. The concept
of nibbling draws attention to another possibility. If a graph of order
$n$ can be
nibbled away to zero then it has a sequence of isometric induced subgraphs
of every order $\le n$, each of which is an induced subgraph of the
one before. We have given evidence 
that most graphs have this property.

\section{Applications}

The concepts of distance redundancy and strong and weak graphs have some
modest application to the design of communications networks as was
already remarked. The weak graphs have the property that every vertex is
playing its part in providing some shortest route. The strong graphs have
the property that distances are protected under the elimination of any
given vertex. It may also be useful in some circumstances to compute which
if any vertices are distance-redundant in some given graph. For example,
if one wishes to compute all distances in the graph, it may be useful
on a first pass to discover which vertices are first-order
distance-redundant, 
and then subsequently
to leave those vertices out of the computation when finding distances between other pairs (paying attention to the conditions stipulated in lemma
\ref{l.isored} and theorem \ref{th.isosurround}). If one 
has reason to believe the graph has high diameter, it may also 
be worthwhile to find the
2nd- or higher-order distance-redundant vertices
or to nibble the graph. 

The concept of distance redundancy
readily generalises to the case of a graph with edge weights.
A vertex can be deemed distance-redundant 
if it does not offer
a shorter route between any of its neighbours than is available
without it. A vertex which offers shortest routes among all
the non-adjacent neighbours of $v$ may be said to surround $v$.

\appendix
\section{Further remarks on graph generation}

\lemm{\label{l.nored}
If graph $G$ has no redundant vertices then 
in $G-v$ no vertex is surrounded by a redundant vertex, 
$\forall v  \in G$.}

\begin{proof}
If $G-v$ has a vertex surrounded by a redundant vertex, then in order
to form $G$ with no redundant vertices, one would have to introduce a
vertex $v$ which is adjacent to the surrounded vertex $u$ (otherwise that vertex
will still be surrounded) and which is also adjacent to the 
surrounding
redundant vertex
$w$ (otherwise that vertex will still be redundant). But after this
operation, $w$ will be adjacent to all of $N(u) \setminus \{w\}$ so $u$ is surrounded.
This contradicts the requirement on $G$. \qed
\end{proof}

\begin{corollary}
In order to generate all graphs of order $n$ containing no
redundant vertices, it is sufficient to consider just those which can be
obtained by augmenting (by the addition of a vertex and edges incident on it)
graphs of order $n-1$ in which no vertex is surrounded by a redundant vertex.
\end{corollary}

\ifodd 0
\begin{table}
\begin{tabular}{c|rrr}
$n$ & no vsr. & wk,2 & wk \& $d_{\rm min}$ \\
\hline
2 & 0 & 2 & 2\\ 
3 & 0 & 4 & 4\\
4 & 1 & 10 & 7\\
5 & 3 & 29 & 18\\
6 & 13 & 114 & 62\\
7 & 70 & 664 & 336\\
8 & 788 & 6754 & 2357\\
9 & ? & \\
10 & ? & \\
\end{tabular}
\caption{Statistics of small graphs. The columns indicate the
number of graphs of order $n$ with, respectively:
no vertex surrounded by a redundant vertex (the condition of lemma
\ref{l.nored}), 
all vertices in weak links
are within distance 2 of one another (the condition of lemma \ref{l.weak}),
and (last column) that condition with also the condition $|S| \le d_{\rm min}+1$.}
\label{t.statnsr}
\end{table}
\fi

\lemm{\label{l.weak}
If graph $G$ has all vertices redundant then 
$\forall v  \in G$, $G-v$ has the following property:
if $S$ is the set of vertices which are weakly linked to at
least one other vertex, then all distances
among members of $S$ are $\le 2$.}

\begin{proof}
We consider a general graph $G-v$ and seek to add a single vertex,
and edges incident on that vertex, such that the resulting graph will have
no weakly linked pairs. First consider a weakly-linked vertex pair $(u,w)$
in $G-v$. The new vertex has to be adjacent to both members of this pair
or they will still be weakly linked in $G$. If that is the only such pair
then the added vertex is a twin so $G$ has all vertices redundant, as
required. If there are further weakly linked pairs, consider any one of them
$(s,t)$ (this pair may or may not intersect the pair $(u,w)$.) 
The added vertex $v$ must be adjacent to both of $s$ and $t$
in order to strengthen their link. But $v$ now provides a 2-walk from
each of $(u,w)$ to each of $(s,t)$. These will be weak links in $G$ unless
in $G-v$ each of $(u,w)$ is at distance less than or equal to 2 from
each of $(s,t)$. The argument applies to all pairs of pairs in the set 
$S$.
 \qed
\end{proof}

\begin{corollary}
In order to generate all graphs of order $n$ in which all vertices
are redundant, it is sufficient to consider just those which can be
obtained by augmenting (by the addition of a vertex and edges incident on it)
graphs of order $n-1$ in which the property described in the lemma holds.
\end{corollary}

\bibliographystyle{plain}

\bibliography{graphrefs}

\end{document}